\documentclass[11pt]{article}

\usepackage{geometry}
\usepackage{color}
\usepackage{float}
\usepackage{textcomp}
\usepackage{amsthm}
\usepackage{amsmath}
\usepackage{amssymb}

\usepackage{changes}

\usepackage{mathtools}

\newtheorem{theorem}{Theorem}[section]
\newtheorem{lemma}[theorem]{Lemma}
\newtheorem{corollary}[theorem]{Corollary}
\newtheorem{proposition}[theorem]{Proposition}

\newtheorem{definition}[theorem]{Definition}
\newtheorem{example}[theorem]{Example}

\DeclareMathOperator{\argmin}{argmin}

\DeclareMathOperator{\dom}{dom}

\newcommand{\R}{\mathbb{R}}

\newcommand{\Dom}{\mathrm{Dom}}

\newcommand{\inner}[2]{\langle{#1},{#2}\rangle}

\newcommand{\norm}[1]{\|#1\|}

\newcommand{\ri}{{\mbox{\rm ri\,}}}
\newcommand{\tos}{\rightrightarrows} 
\newcommand{\Y}{\mathcal{Y}}
\newcommand{\X}{\mathcal{X}}
\newcommand{\Z}{\mathcal{Z}}

\newcommand{\Sf}{\mathcal{S}}
\newcommand{\inte}{\mathrm{int}}
\newcommand{\D}{\mathcal{D}}

\newcommand{\vgap}{\vspace{.1in}}

\newcommand{\tx}{\tilde x}

\newcommand{\tz}{\tilde z}

\newcommand{\bi}{\begin{itemize}}
\newcommand{\ei}{\end{itemize}}
\newcommand{\ba}{\begin{array}}
\newcommand{\ea}{\end{array}}

\newcommand{\mgap}{\vspace{.1in}}

\begin{document}

\title{ Improved pointwise iteration-complexity of a regularized ADMM and of a regularized non-Euclidean HPE framework}

\author{
    Max L.N. Gon\c calves
    \thanks{Institute of Mathematics and Statistics, Federal University of Goias, Campus II- Caixa
    Postal 131, CEP 74001-970, Goi\^ania-GO, Brazil. (E-mails: {\tt
       maxlng@ufg.br} and {\tt jefferson@ufg.br}). This work was done while these authors were a postdoc at the School of Industrial and Systems Engineering, Georgia Institute of Technology, Atlanta, GA.  The work of these authors was
    supported in part by  CNPq Grant  444134/2014-0, 309370/2014-0, 200852/2014-0, 201047/2014-4  and FAPEG/GO.}
    \and
      Jefferson G. Melo \footnotemark[1]
    \and
    Renato D.C. Monteiro
    \thanks{School of Industrial and Systems
    Engineering, Georgia Institute of
    Technology, Atlanta, GA, 30332-0205.
    (email: {\tt monteiro@isye.gatech.edu}). The work of this author
    was partially supported by NSF Grant CMMI-1300221.}
}
\date{January 4, 2016 (Revised in August 22, 2016)}

\maketitle

\begin{abstract}
This paper describes a regularized variant of the alternating direction method of multipliers (ADMM) for solving
linearly constrained convex  programs. It is shown that the pointwise iteration-complexity of the new variant is better
than the corresponding one for the standard ADMM method and that, up to a logarithmic term, is identical to the ergodic 
iteration-complexity of the latter method. Our analysis is based on first presenting and establishing
the pointwise iteration-complexity of a regularized non-Euclidean hybrid 
proximal extragradient 
framework whose error condition at each iteration
includes both a relative error and a summable error.
It is then shown that the new ADMM variant is a special instance of the latter framework where the sequence of summable errors is
identically zero when the ADMM stepsize is less than one or a nontrivial sequence when the stepsize is
in the interval $[1,(1+\sqrt{5})/2)$.
 \end{abstract}

%
%
%
%
%
%
%
%
%
%
%
%
%
%
%
%
%
%
%
%
%
%
%
%
%

\pagestyle{plain}

\section{Introduction} \label{sec:int}
The goal of this paper is to present a  regularized variant of the alternating direction method of multipliers (ADMM) for solving
the linearly constrained convex problem
\begin{equation} \label{optl1}
\inf \{ f(y) + g(s) : C y + D s = c \}
\end{equation}
where $\X$, $\Y$ and $\Sf$ are  finite dimensional inner product spaces, 
$f: \Y \to (-\infty,\infty]$ and $g:\Sf \to (-\infty,\infty]$ are proper
closed convex functions,  $C: \Y \to \X$ and $D: \Sf \to \X$ are linear operators,
and $c \in \X$.  Many methods have been proposed to solve problems with separable structure such as \eqref{optl1} (see for example \cite{ ACN_ref2, BAC_ref2, ChPo_ref2,DST_ref2,GADMM2015,FPST_editor,0352.65034,0368.65053,Goldstein2014, He2,monteiro2010iteration, LanADMM,PJX_ref2} and the references cited
therein).

A well-known class of ADMM instances for solving \eqref{optl1} recursively computes a sequence
$\{(s_k,y_k,x_k)\}$ as follows. 
Given $(s_{k-1},y_{k-1},x_{k-1})$, the $k$-th triple $(s_k,y_k,x_k)$ is determined as
\begin{align}
s_k &= \argmin_{s} \left \{ g(s) - \inner{ {x}_{k-1}}{Ds}_\X +
\frac{\beta}{2} \| C y_{k-1} + D s - c \|^2_\X + \frac12 \inner{ s-s_{k-1}}{H(s-s_{k-1})} \right\},\nonumber\\
y_k &= \argmin_y \left \{ f(y) - \inner{x_{k-1}}{Cy}_\X +
\frac{\beta}{2} \| C y + D s_k - c \|^2_\X+\frac12 \inner{ y-y_{k-1}}{G(y-y_{k-1})} \ \right\},\label{ADMMclass}\\
x_k &= x_{k-1}-\theta\beta\left[Cy_k+Ds_k-c\right] \nonumber
\end{align}
where $\beta > 0$ is a fixed penalty parameter, $\theta>0$ is a fixed stepsize and
$H$, $G$ are  fixed positive semidefinite self-adjoint linear operators.
If $(H,G)=(0,0)$ in the above class, we obtain the standard ADMM. 

The ADMM   was  introduced in \cite{0352.65034,0368.65053} and is thoroughly discussed in \cite{Boyd:2011,glowinski1984}.
Recently, there has been some growing interest in ADMM (see  for  instance \cite{Boley2013,Cui,Eckstein1,Goldstein2014,Hager,LanADMM,yin2015} and the references
cited therein).
To discuss the complexity results about ADMM, we use the terminology
weak pointwise or strong pointwise bounds
to refer to complexity bounds relative to the best of the $k$ first iterates or the last iterate, respectively, to
satisfy a suitable termination criterion.
The first iteration-complexity bound for the ADMM was established only recently in  
\cite{monteiro2010iteration} under the assumption that $C$ is injective.
More specifically, the ergodic iteration-complexity for the standard ADMM  is derived in  \cite{monteiro2010iteration} for any $\theta \in (0,1]$  while
a weak pointwise iteration-complexity easily follows from the approach in \cite{monteiro2010iteration} for any $\theta \in (0,1)$.
Subsequently, without assuming that $C$ is injective, \cite{HeLinear} established the ergodic iteration-complexity of the  ADMM class \eqref{ADMMclass}
with $G=0$ and $\theta=1$ and, as a consequence, of the well-known split inexact Uzawa method~\cite{Xahang} which chooses
$H=\alpha I - \beta D^*D$ for some $\alpha \ge \beta \|D\|^2$.
Paper \cite{He2} establishes the weak pointwise and ergodic iteration-complexity  
of another collection of ADMM instances which includes the standard ADMM for any $\theta \in (0,(1+\sqrt{5})/2)$.
A strong pointwise iteration-complexity bound for the  ADMM class  \eqref{ADMMclass} with $G=0$ and $\theta =1$ is derived in~\cite{He2015}.
Finally, a number of papers (see for example \cite{Cui,Deng1,GADMM2015,Gu2015,Hager,Lin} and  references therein) have extended
most of these complexity results to the context of the ADMM class \eqref{ADMMclass} as well as other more general ADMM classes.

Although different termination criteria are used in the aforementioned papers, their complexity results can be easily rephrased
in terms of a simple termination, namely: for a given $\rho>0$, terminate with a quadruple $(s,y,x,x') \in \Sf \times \Y\times \X \times \X$ satisfying
\[
\max \{ \|Cy + Ds -c\| , \| x' - x\| \} \le \rho, \quad 0 \in \partial_{\rho} g(s) - D^* x, \quad 0 \in \partial_{\rho} f(y) - C^* x'.
\]
In terms of this termination, the best pointwise iteration-complexity bounds are ${\cal O}(\rho^{-2})$ while the best ergodic
ones are ${\cal O}(\rho^{-1})$ but the pointwise results guarantee that above two inclusions hold with $\rho=0$
(i.e., with $\partial_\rho$ replaced by $\partial$).
This paper  presents a regularized variant of the  ADMM  class \eqref{ADMMclass} whose strong pointwise iteration-complexity is
${\cal O} (\rho^{-1}\log (\rho^{-1}))$ for any stepsize $\theta \in (0,(1+\sqrt{5})/2)$. Note that the latter complexity is better
than the pointwise iteration complexity for the  class  \eqref{ADMMclass} by an ${\cal O}(\rho \log(\rho^{-1}))$ factor.

It was shown in \cite{monteiro2010iteration}
that the standard ADMM with  $\theta \in (0,1]$ and $C$ injective can be viewed as an inexact proximal point (PP) method, more specifically, 
as an instance of the hybrid proximal extragradient (HPE) framework proposed by \cite{Sol-Sv:hy.ext}. 
In contrast to the original Rockafellar's PP method which is based on a summable error condition,
the HPE framework is 
based on a relative HPE error condition
involving Euclidean distances.
Convergence results for the HPE framework are studied in
\cite{Sol-Sv:hy.ext}, and  its weak pointwise and ergodic iteration-complexities are established in \cite{monteiro2010complexity} (see also \cite{monteiro2011complexity,monteiro2010iteration}). 
 Applications of the HPE framework to the iteration-complexity analysis
of several zero-order (resp., first-order)
methods  for solving monotone variational
inequalities and monotone inclusions (resp.,  saddle-point problems)
are discussed in \cite{YHe2,YHe1,monteiro2010complexity,monteiro2011complexity,monteiro2010iteration}.
Paper \cite{SvaiterBregman} describes and studies the convergence of a non-Euclidean HPE (NE-HPE) framework which essentially generalizes
the HPE one to the context of general Bregman distances. The latter framework was further generalized in \cite{Oliver} where its
ergodic iteration-complexity was established.
More specifically, consider the monotone inclusion problem
$0 \in T(z)$
where $T$ is a maximal monotone operator and let $w$ be a convex  differentiable  function.
Recall that for a given pair $(z_-,\lambda)=(z_{k-1},\lambda_k)$, the exact PP method
computes the next iterate $z=z_k$ as the (unique) solution of the  prox-inclusion
$\lambda^{-1} [\nabla w(z_{-})-\nabla w(z)] \in T(z)$.
An instance of the NE-HPE framework described in \cite{Oliver} computes an approximate solution
of this inclusion based on the following relative NE-HPE  error criterion:
for some tolerance $\sigma \in [0,1]$, a triple $(\tilde z,z,\varepsilon)=(\tilde z_k,z_k,\varepsilon_k)$ is computed  such that
\begin{equation}\label{eq:appr-sol_bregman}
r:=\frac{1}{\lambda}\left[\nabla w(z_{-})-  \nabla w(z)\right]   \in T^{\varepsilon}(\tilde{z}), \quad (dw)_{z}(\tilde{z}) + \lambda \varepsilon \leq \sigma (dw)_{z_{-}}(\tilde{z})
\end{equation}
where $dw$ is the Bregman distance defined as
$(dw)_z(z')= w(z')-w(z)-\inner{\nabla w(z)}{z'-z}$ for every $z, z'$  and $T^{\varepsilon}$ denotes the $\varepsilon$-enlargement~\cite{Bu-Iu-Sv:teps}
of $T$ (it has the property that $T^{\varepsilon}(u) \supset T(u)$ for each $u$ with
equality holding when $\varepsilon=0$).
Clearly, if $\sigma=0$ in \eqref{eq:appr-sol_bregman}, then
$z=\tz$ and $\varepsilon=0$, and the inclusion in \eqref{eq:appr-sol_bregman} reduces to the prox-inclusion.
Also, the HPE framework  is the special case of the NE-HPE one in which $w(\cdot)=\|\cdot\|^2/2$ and $\|\cdot\|$ is the Euclidean norm.

Section~\ref{sec:smhpe} considers a monotone inclusion problem (MIP) of the form $0 \in (S+T)(z)$ where $S$ and $T$ are maximal monotone, $S$ is $\mu$-monotone with respect to $w$ for some $\mu>0$ (see condition {\bf A1})
and $w$ is a regular distance generating function (see Definition \ref{def:assu}).
It then presents and  establishes the strong pointwise iteration-complexity of a variant of the NE-HPE framework for solving such a MIP
in which the inclusion in \eqref{eq:appr-sol_bregman} is strengthened to $r \in S(\tilde z) + T^{\varepsilon}(\tilde{z})$ but its
error condition is weakened in that
an additional nonnegative tolerance is added to the right hand side of the inequality in \eqref{eq:appr-sol_bregman} which is $\tau$-upper summable.
This extension of the error condition will be useful in the analysis of the regularized ADMM  class of
Section~\ref{sec:amal} with ADMM stepsize $\theta>1$.

Section~\ref{sec:imp.hpe} presents and establishes the strong pointwise  iteration-complexity of a 
regularized NE-HPE framework which solves
the inclusion $0 \in T(z)$ where $T$ is maximal monotone. The latter framework is based on the idea of  invoking  the
above  NE-HPE variant to solve perturbed MIPs of the form $0 \in (S+T)(z)$ where $S(\cdot)= \mu [ \nabla w (\cdot) - \nabla w(z_0)]$ for some $\mu>0$, point $z_0$
and regular distance generating function $w$.

Section \ref{sec:amal} presents and establishes the ${\cal O} (\rho^{-1}\log (\rho^{-1}))$ strong pointwise iteration-complexity of
a regularized ADMM class whose description depends  on $\beta$, $\theta$ (as the standard ADMM) and a regularization parameter $\mu$.
It is well-known that \eqref{optl1} can be reformulated as a monotone inclusion problem of the form $0 \in T(z)$ with $z=(s,y,x)$.
The regularized ADMM class can be viewed as a special instance of the regularized NE-HPE framework applied to the latter inclusion where:
i)  all stepsizes $\lambda_k$'s are equal one; ii) the distance generating function $w$ depends on $\beta$, $\theta$ and operator $C$ as in relation~\eqref{df:norm_admm};
and, iii) the sequence of $\tau$-upper summable errors is
zero when the ADMM stepsize $\theta \in (0,1)$
and nontrival (and hence nonzero) when $\theta \in [1,(1+\sqrt{5})/2)$.
Hence, the iteration complexity analysis of the regularized ADMM class  for the case in which $\theta \in [1,(1+\sqrt{5})/2)$ requires both
a combination of relative and $\tau$-upper summable errors while the one for the case of $\theta \in (0,1)$ requires only relative errors.
Moreover, the distance generating function $w$ is strongly convex only when $C$ is injective but is always regular and hence fulfills the conditions required for the
iteration-complexity results of Section \ref{sec:imp.hpe} to hold. 

This paper is organized as follows. 
 Subsection~\ref{sec:bas} presents the notation and review some basic concepts about convexity
and maximal monotone operators.
 Section~\ref{sec:smhpe}  introduces the  class of regular distance generating functions and 
presents  the aforementioned variant of the NE-HPE framework.
Section~\ref{sec:imp.hpe} presents the regularized NE-HPE framework and its complexity analysis.
Section~\ref{sec:amal} contains two subsections. 
Subsection~\ref{subsec:Admm1} describes the regularized ADMM class and  its  pointwise iteration-complexity result whose proof is
given in  Subsection~\ref{subsec:proofAdmm}.
Finally, the appendix reviews some basic  results about dual seminorms and existence of optimal solutions and/or Lagrange multipliers for
linearly constrained convex programs, and presents the proofs of one result of Section~\ref{sec:smhpe} and two results of Subsection~\ref{subsec:proofAdmm}.

\subsection{Basic concepts and notation}
\label{sec:bas}

This subsection presents some definitions, notation and terminology  needed by our presentation.

The set of real numbers is denoted by $\mathbb{R}$. The set of non-negative real numbers  and 
the set of positive real numbers are denoted by $\R_+$ and $\R_{++}$, respectively. 
For $t>0$, we let $\log^+(t):=\max\{\log t,0\}$.

Let  $\Z$ be a finite-dimensional 
real vector space with
inner product denoted by $\inner{\cdot}{\cdot}$ and let $\|\cdot\|$ denote an arbitrary
seminorm in $\Z$.  Its dual (extended) seminorm, denoted by $\|\cdot\|^*$, is defined as
 $\|\cdot\|^*:=\sup\{ \inner{\cdot}{z}:\|z\|\leq 1\}$.  
 Some basic properties of the dual seminorm are given in Proposition~\ref{propdualnorm} in Appendix~\ref{basiresul}. The interior and the relative interior of a set $U\subset \Z$ are denoted, respectively, by $\mbox{int}(U)$ and $ \mbox{ri}(U)$ (see for example pp. 43-44 of \cite{Rockafellar70} for their definitions).

Given a set-valued operator $S:\Z\tos \Z$,
its domain is denoted by
 $\mbox{Dom}(S):=\{z\in \Z\,:\, S(z)\neq \emptyset\}$
 and  its {inverse} operator $S^{-1}:\Z\tos \Z$ is given by
$S^{-1}(v):=\{z\;:\; v\in S(z)\}$.
The operator $S$ is said to be   monotone if 
\[
\inner{z-z'}{s-s'}\geq 0\quad \forall \; z,z'\in Z, \forall \;s\in S(z), \forall \; s'\in S(z').
\]
Moreover, $S$ is maximal monotone if it is monotone and, additionally, if $T$ is a monotone operator such that $S(z)\subset T(z)$ for every $z\in \Z$  then $S=T$.
The sum $ S+T:\Z\tos \Z$ of two set-valued 
operators $S,\, T:\Z\tos \Z$ is defined by
$
(S+T)(x):=
\{a+b\in \Z\,:\,a\in S(x),\; b\in T(x)\}
$
for every $z\in \Z$.
Given a scalar $\varepsilon\geq0$, the 
 {$\varepsilon$-enlargement} $T^{[\varepsilon]}:\Z\tos \Z$
 of a monotone operator $T:\Z\tos \Z$ is defined as
\begin{align}
\label{eq:def.eps}
 T^{[\varepsilon]}(z)
 :=\{v\in \Z\,:\,\inner{v-v'}{z-z'}\geq -\varepsilon,\;\;\forall z' \in \Z, \forall\; v'\in T(z')\} \quad \forall z \in \Z.
\end{align}

Recall that the 
{$\varepsilon$-subdifferential} of a 
 convex function $f:\Z\to  [-\infty,\infty]$
is defined by
$\partial_{\varepsilon}f(z):=\{v\in \Z\,:\,f(z')\geq f(z)+\inner{v}{z'-z}-\varepsilon\;\;\forall z'\in \Z\}$ for every $z\in \Z$.
When $\varepsilon=0$, then $\partial_0 f(x)$ 
is denoted by $\partial f(x)$
and is called the {subdifferential} of $f$ at $x$.
The operator $\partial f$ is trivially monotone if $f$ is proper.
If $f$ is a proper lower semi-continuous convex function, then
$\partial f$ is maximal monotone~\cite{Rockafellar}.
The domain of $f$ is denoted by $\dom f$ and 
the conjugate of $f$ is the function
$f^*:\Z \to [-\infty,\infty]$ defined as
\[
f^*(v) = \sup_{z \in \Z} \left(\inner{v}{z} - f(z) \right)\quad \forall v \in \Z.
\]

\section
{A non-Euclidean HPE framework for  a special class of MIPs}
\label{sec:smhpe}
This section describes and derives convergence rate bounds for a non-Euclidean HPE framework for solving inclusion problems
consisting of the sum of two maximal monotone operators, one of which is
assumed to be $\mu$-monotone with respect to a Bregman distance for some $\mu>0$.
The latter concept implies strong monotonicity of the operator when the Bregman distance is nondegenerate, i.e.,
its associated distance generating function is strongly monotone.  However, it should be noted that when the Bregman distance
is degenerate, the latter concept does not imply strong monotonicity of the operator.

We start by introducing the definition of a distance generating function and its corresponding
Bregman distance adopted in this paper.

\begin{definition}\label{defw0}
A proper lower semi-continuous convex function $w : \Z  \to [-\infty,\infty]$ is called a distance generating function
if   $\inte(\dom w) = \Dom (\partial w) \neq \emptyset$ and $w$ is continuously differentiable on this interior.
Moreover, $w$ induces the Bregman distance $dw: \Z \times \inte(\dom w) \to \mathbb{R}$ defined as
\begin{equation}\label{def_d}
(dw)(z';z) := w(z')-w(z)-\langle \nabla w(z),z'-z\rangle  \quad \forall (z', z) \in \Z \times \inte(\dom w).
\end{equation}
\end{definition}
For simplicity, for every $z \in \inte(\dom w)$,  the function $(dw)(\,\cdot\,;z)$ will be denoted by $(dw)_{z}$ so that
\[
(dw)_{z}(z')=(dw)(z';z) \quad \forall (z',z) \in \Z \times \inte(\dom w).
\]
The following useful identities follow straightforwardly from \eqref{def_d}:
\begin{align}
\nabla (dw)_{z}(z') &= - \nabla (dw)_{z'}(z) = \nabla w(z') - \nabla w(z) \quad \forall z, z' \in \inte(\dom w), \label{grad-d} \\
\label{equacao_taylor}
(dw)_{v}(z') - (dw)_{v}(z) &= \langle \nabla (dw)_{v}(z), z'-z\rangle + (dw)_{z}(z') \quad \forall z' \in \Z, \, \forall v  , z \in \inte(\dom w).
\end{align}

Our analyses of the non-Euclidean HPE frameworks presented in Sections \ref{sec:smhpe} and \ref{sec:imp.hpe}
require an extra property of the distance generating function, namely, that of being regular
with respect to a seminorm.

\begin{definition}\label{def:assu}
Let distance generating function $w:\Z \to [-\infty,\infty]$, seminorm $\|\cdot\|$ in $\Z$ and convex set $Z \subset \inte(\dom w)$
be given.
For given positive constants $m$ and $M$,   $w$ is said to be $(m,M)$-regular with respect to $(Z,\|\cdot\|)$ if 
\begin{align}\label{strongly}
 (dw)_{z}(z')\geq \frac{{m}}{2}\|z-z'\|^2 \quad \forall  z,z' \in  Z,
\end{align}
\begin{equation}\label{a1}
\|\nabla w(z)-\nabla w(z')\|^*\leq {M}\|z-z'\| \quad \forall z, z' \in   Z.
\end{equation}
\end{definition}

Note that if the seminorm in Definition \ref{def:assu} is a norm, then \eqref{strongly}  implies that
$w$ is strongly convex, in which case the corresponding $dw$ is said to be nondegenerate.
However, since we are not necessarily assuming that $\|\cdot\|$ is a norm, our approach includes the
case of $w$ being not strongly convex, or equivalently, $dw$ being degenerate
(e.g., see Example~\ref{arth2}(b) below).

It is worth pointing out that if  $w:\Z \to [-\infty,\infty]$ is  $(m,M)$-regular with respect to $(Z,\|\cdot\|)$, 
then
\begin{equation}\label{lipsc}
\frac{m}{2}\|z-z'\|^2\leq (d{w})_{z}(z')  \leq \frac{M}{2}\|z-z'\|^2  \quad \forall z,z' \in Z.
\end{equation}

Some examples of regular distance generating functions are as follows.

\begin{example}\label{arth2}
a) The distance generating function  $ w: \Z \to \mathbb{R}$ 
defined by $ w(\cdot):=\inner{\cdot}{\cdot}/2$ is a $(1,1)$-regular with respect to 
$(\Z,\|\cdot\|)$ where $\|\cdot\|:=\inner{\cdot}{\cdot}^{1/2}$.
 \\[2mm]
 b)  Let  $A:\Z \to {\Z}$ be  a self-adjoint positive semidefinite linear operator. 
 The distance generating function  $ w: \Z \to \mathbb{R}$ defined by 
 $ w(\cdot):=\inner{A(\cdot)}{\cdot}/2$ is a $(1,1)$-regular with respect to 
 $(\Z,\|\cdot\|)$ where $\|\cdot\|:=\inner{A(\cdot)}{\cdot}^{1/2}$.
 \\[2mm]
c) Let $\delta\in(0,1]$ be given  and  define  $W:=\{x\in \mathbb{R}^n:  x_i+\delta/n>0, \, \forall i=1,\ldots,n\}$.
Let the distance generating function  $w:\R^n \to [-\infty,\infty]$   defined by  
  $ w(x):=\sum_{i=1}^{n}(x_i+\delta/n)\log(x_i+\delta/n)$ for every $x \in W$ and $w(x):=\infty$ otherwise.
  Then, $w$ is 
a $(1/(1+\delta),n/\delta)$-regular with respect to 
 $(Z,\|\cdot\|_1)$ where $Z=\{x\in \mathbb{R}^n: \sum_{i=1}^{n} x_i=1, \, x_i\geq 0, \, i=1,\ldots,n\}$.
\end{example}

The following result gives some useful properties of regular distance generating functions.

\begin{lemma}\label{basicassu}
Let  $w:\Z \to [-\infty,\infty]$ be an $(m,M)$-regular distance generating function with respect to $(Z,\|\cdot\|)$ as in Definition \ref{def:assu}.
Then,
\begin{equation}\label{eq:56}
\left(1+\frac{1}{t}\right)^{-1} (dw)_{z}(z') \leq \frac{M}{m} \left[ (dw)_{z}(\tilde z)+t(dw)_{\tilde z}(z') \right]   \quad \forall t >0, \;   \forall\, z, z', \tilde{z} \in Z;
\end{equation}
\begin{equation}\label{eq:789}
\|\nabla(dw)_{z'}(z)\|^*\leq  \frac{\sqrt{2}M}{\sqrt{m}} [(dw)_{z'}(z)]^{1/2} \quad \forall\, z, z' \in Z.
\end{equation}
\end{lemma}
\proof
To show \eqref{eq:56}, let $t>0$ and $z,  z', \tilde z \in Z$ be given.  By \eqref{strongly}, we have
\begin{align}\label{a3}
(dw)_{z}(\tilde z)+t(dw)_{\tilde z}(z')\geq  \frac{m}{2}\left(\|z-\tilde z\|^2+t\|\tilde z-z'\|^2\right).
\end{align}
Using the fact that
 \begin{align*}
    \min_{\gamma_1,\gamma_2}\{\gamma_1^2+t\gamma_2^2\;|\; \gamma_1,\gamma_2 \geq 0,\;
    \gamma_1 + \gamma_2 \geq \|z-z'\| \}= (1+1/t)^{-1} \|z-z'\|^2
 \end{align*}
 and $(\gamma_1, \gamma_2) = (\|z-\tilde z\|,\|\tilde z-z'\|)$ is a feasible point for the above problem,
 we then conclude that
 \[
 \|z-\tilde z\|^2+t\|\tilde z-z'\|^2\geq(1+1/t)^{-1} \|z-z'\|^2
 \]
 which, together with the second inequality in \eqref{lipsc} 
and  \eqref{a3}, immediately  yields  \eqref{eq:56}.
Finally, it is easy to see that \eqref{eq:789} immediately follows from
\eqref{grad-d},  \eqref{strongly} and \eqref{a1}. 
\endproof

Throughout this section, we assume that, for some positive scalars $m$ and $M$,
 $w:\Z \to [-\infty,\infty]$ is an $(m,M)$-regular distance generating function 
with respect to $(Z,\|\cdot\|)$ where $Z \subset \inte(\dom w)$ is a convex set and 
$\|\cdot\|$ is a seminorm in~$\Z$.
Our problem of interest in this section is the MIP
\begin{align}\label{eq:inc.p}
 0\in (S+T)(z)
\end{align}
where
$S,T:\Z\tos \Z$ are point-to-set operators satisfying the following conditions:
\begin{itemize}
\item[\bf A0)] $S$ and $T$ are maximal monotone and $\Dom(T)\subset Z$;  
\item[\bf A1)]  $S$ is $\mu$-monotone on $Z$ with respect to $w$, i.e., there exists a constant $\mu>0$ such that  
\begin{align}\label{eq:sm.c}
 \inner{z-z'}{s-s'}\geq \mu[ (dw)_{z}(z')+ (dw)_{z'}(z)] \quad \forall \, z,z' \in Z, \forall \,s\in S(z), \forall \, s'\in S(z');
\end{align}
\item[\bf A2)]  the solution set $(S+T)^{-1}(0)$ of~\eqref{eq:inc.p} is nonempty. 
\end{itemize}


We observe that when the seminorm $\|\cdot\|$ is a norm, then
\eqref{eq:sm.c} implies that  $S$ is strongly monotone. 
 However, the latter needs not be the case when $\|\cdot\|$ is not a norm.

We now state a non-Euclidean-HPE (NE-HPE) framework  for solving \eqref{eq:inc.p}  which generalizes the ones studied in  \cite{Oliver,SvaiterBregman}.

\mgap
\mgap

\noindent
\fbox{
\begin{minipage}[h]{6.4 in}
{\bf Framework~1} { (A NE-HPE variant for solving \eqref{eq:inc.p})}.
\begin{itemize}
\item[(0)] Let $z_0 \in Z$, $\eta_0\geq 0$,   $\sigma \in [0, 1)$, $\tau \in (0,1)$ and $\lambda\in \R_{++}$ be given, and set $k=1$;
\item[(1)] choose $\lambda_k\geq \lambda$ and find $(\tilde{z}_k, z_k, \varepsilon_k,\eta_k) \in Z \times Z \times \mathbb{R}_{+}\times \mathbb{R}_{+}$   such that 
       \begin{align}
& r_k:= \frac{1}{\lambda_k} \nabla (dw)_{z_k}(z_{k-1})  \in \left(S+T^{[\varepsilon_k]}\right)(\tz_k), \label{subpro} \\      
& (dw)_{z_k}({\tz}_k) + \lambda_k\varepsilon_k +\eta_k \leq \sigma (dw)_{z_{k-1}}({\tz}_k)+(1-\tau)\eta_{k-1}; \label{cond1}
\end{align}
       
\item[(2)] set $k\leftarrow k+1$ and go to step 1.
\end{itemize}
\noindent
{\bf end}
\end{minipage}
}

\vgap
\vgap

We now make some remarks about Framework~1. First, it does not specify
how to find $\lambda_k$ and $(\tilde{z}_k, z_k, \varepsilon_k,\eta_k)$ satisfying (\ref{subpro}) and (\ref{cond1}). The particular 
scheme for computing $\lambda_k$ and $(\tilde{z}_k, z_k, \varepsilon_k,\eta_k)$ will depend on the instance of the framework under consideration
and the properties of the operators  $S$ and $T$.
Second, if $w$ is strongly convex on $Z$, $\sigma= 0$ and $\eta_0=0$, then (\ref{cond1}) implies that $\varepsilon_k= 0$, $\eta_k=0$ and $z_k = \tilde z_k$ for every~$k$, and
hence that $r_k \in (S+T)(z_k)$ in view of \eqref{subpro}.
Therefore, the HPE error conditions \eqref{subpro}-\eqref{cond1} can be viewed as a relaxation of an iteration of the exact non-Euclidean proximal point method,
namely,
\begin{equation} \label{eq:exact-prox}
0 \in \frac{1}{\lambda_k} \nabla (dw)_{z_{k-1}}(z_{k})  +  (S+T)({z}_k).
\end{equation}
Third, if $w$ is strongly convex on $Z$ and $S+T$ is maximal monotone, then Proposition~\ref{existence} with $T=\lambda_k(S+T)$ and  $\hat z=z_{k-1}$
implies
 that the above inclusion has a unique solution $z_k$,
and hence that, for any given $\lambda_k>0$, there exists a quadruple $(\tz_k,z_k,\varepsilon_k,\eta_k)$ of the form $(z_k,z_k,0,0)$ satisfying
(\ref{subpro})-(\ref{cond1}) with $\sigma=0$ and $\eta_{k-1}=0$. 
Considering inexact quadruples (i.e., those satisfying the HPE relative error conditions (\ref{subpro})-(\ref{cond1}) with $\sigma>0$)
other than an exact one (i.e., one of the form $(z_k,z_k,0,0)$ where $z_k$ is a solution of \eqref{eq:exact-prox}) has important implications, namely:
i) the resulting HPE framework contains a variety of methods for convex programming, variational inequalities and saddle points
as special instances (see for example \cite{YHe2,YHe1,monteiro2010complexity,monteiro2011complexity,monteiro2010iteration,Sol-Sv:hy.ext}) and;
ii) it provides much greater computational flexibility since finding the exact quadruple is impossible for most MIPs.
Fourth, even though  the definition of a regular distance generating function
   does not exclude the case in which
$w$ is constant, such a case is not interesting from an algorithmic analysis point of view. In fact, if $\eta_0=0$ and
$w$ is constant, then we have that $\tilde z_1$ is already a solution of~\eqref{eq:inc.p} since it follows from \eqref{cond1} with $k=1$ that $\varepsilon_1=0$, and hence that
$0 \in (S+T)(\tilde z_1)$ in view of~\eqref{subpro} with $k=1$.
Fifth, the more general HPE error condition \eqref{cond1} is clearly equivalent to
\[
(dw)_{z_k}(\tilde z_k) + \lambda_k \varepsilon_k \le \sigma (dw)_{z_{k-1}}(\tilde z_k) + \tilde \eta_k
\]
where $\tilde \eta_k = (1-\tau) \eta_{k-1}-\eta_k$. The consideration of this additional error $\{\tilde \eta_k\}$ will be useful in the analysis of the regularized 
ADMM class studied in Section \ref{sec:amal}.
Observe that  $\{\tilde \eta_k\}$ is $\xi$-upper summable, i.e.,
\[
\limsup_{k \to \infty} \sum_{j=1}^{k} \frac{ \tilde \eta_j}{(1-\xi)^j}  < \infty
\]
for any $\xi \in [0,\tau]$, since nonnegativity of $\{\eta_k\}$ implies
\[
\sum_{j=1}^{k} \frac{ \tilde \eta_j}{(1-\xi)^j} \leq  \sum_{j=1}^{k}  \left( \frac{\eta_{j-1}}{(1-\xi)^{j-1}}-\frac{\eta_{j}}{(1-\xi)^{j}} \right) 
= \eta_0 - \frac{\eta_{k}}{(1-\xi)^{k}} \le \eta_0 \quad \forall k \ge 1.
\]

We now make some remarks about the relationship of Framework 1 with the ones studied in \cite{Oliver,Maicon, SvaiterBregman}.
First, Framework 1 with $S=0$ and $\{\eta_k\}$ identically zero reduces to the one studied in \cite{Oliver} and also to the one in \cite{SvaiterBregman} if
$\{\varepsilon_k\}$ is chosen to be identically zero. However, unless $w$ is constant, condition {\bf A1} does not allow us to take $S=0$,
and hence the convergence rate results of this section do not apply to the setting of \cite{Oliver},
and hence of \cite{SvaiterBregman}.
Second, in contrast to \cite{Oliver},
the regularity condition on $w$ and the $\mu$-monotonicity of $S$ with respect to $w$
allow us to establish a geometric (pointwise) convergence
rate for the sequence $\{(dw)_{z_k}(z^*) + \eta_k\}$ for any $z^* \in (S+T)^{-1}(0)$ (see Proposition~\ref{lm:tech2}  below).
Third, when $w$ is the usual Euclidean distance generating function as in Example~\ref{arth2}(a) and $\{\eta_k\}$ is identically
zero,  Framework 1 and the corresponding results derived in this section reduce to the ones studied in
Subsection~2.2 of \cite{Maicon}.

We also remark that the special case of Framework 1 in which  $S(\cdot)=\mu \nabla (dw)_{z_0}(\cdot)$ for some $z_0 \in Z$ and
$\mu>0$ sufficiently small will be used in Section~\ref{sec:imp.hpe}
as a way towards solving the inclusion $0 \in T(z)$. The resulting framework can then be viewed as a regularized
NE-HPE framework in the sense that the operator $T$ is slightly perturbed and regularized by
the operator $\mu \nabla (dw)_{z_0}(\cdot)$.

In the remaining part of this section, we focus our attention on establishing  
convergence rate bounds  for the sequence $\{(dw)_{z_k}(z^*) + \eta_k\}$ and the
sequence of residual pairs $\{(r_k,\varepsilon_k)\}$ generated by any instance of Framework~1.
We start by deriving a preliminary technical result.

\begin{lemma}\label{lema_desigualdades}
For every $k \geq 1$, the following statements hold:
\begin{itemize}
\item[(a)] for every $z \in \dom w$, we have
\[(dw)_{z_{k-1}}(z) - (dw)_{z_k} (z) = (dw)_{z_{k-1}}(\tilde{z}_k) - (dw)_{z_k} (\tilde{z}_k) + \lambda_k \langle r_k, \tilde{z}_k - z \rangle;\]
\item[(b)]
 for every $z \in \dom w$, we have
\[(dw)_{z_{k-1}}(z) - (dw)_{z_k}(z)+(1-\tau)\eta_{k-1} \geq (1 - \sigma)(dw)_{z_{k-1}}(\tilde{z}_k) + \lambda_k (\langle r_k, \tilde{z}_k - z\rangle + \varepsilon_k)+\eta_{k};\]
\item[(c)] for every  $z^* \in (S+T)^{-1}(0)$, we have
\[(dw)_{z_{k-1}}(z^*) - (dw)_{z_k}(z^*)+(1-\tau)\eta_{k-1} \geq (1 - \sigma)(dw)_{z_{k-1}}(\tilde{z}_k) + \lambda_k \mu (dw)_{\tilde{z}_k}(z^*)+\eta_{k}.
\]
\item[(d)] for every  $z^* \in (S+T)^{-1}(0)$, we have
\[  (1 - \sigma)(d{w})_{z_{k-1}}(\tilde{z}_k)\leq (d{w})_{z_{k-1}}(z^*)+\eta_{k-1}, \quad  (dw)_{z_{k}}(\tilde z_{k})  \leq  \frac{1}{1 - \sigma}\left[(dw)_{z_{k-1}}(z^*)+\eta_{k-1}\right],
\]
\end{itemize}
\end{lemma}
\proof 
(a) Using (\ref{equacao_taylor}) twice and using the definition of $r_k$ given by (\ref{subpro}), we obtain
\begin{align*}
 (dw)_{z_{k-1}}(z) - (dw)_{z_k}(z) &= (dw)_{z_{k-1}}(z_k) + \langle \nabla (dw)_{z_{k-1}}(z_k), z - z_k\rangle \\
& = (dw)_{z_{k-1}}(z_k) + \langle \nabla (dw)_{z_{k-1}}(z_k), \tilde{z}_k - z_k\rangle + \langle \nabla (dw)_{z_{k-1}}(z_k), z - \tilde{z}_k\rangle \\
& = (dw)_{z_{k-1}}(\tilde{z}_k) - (dw)_{z_k}(\tilde{z}_k) + \langle \nabla (dw)_{z_{k-1}}(z_k), z - \tilde{z}_k\rangle \\
& = (dw)_{z_{k-1}}(\tilde{z}_k) - (dw)_{z_k}(\tilde{z}_k) + \lambda_k\langle r_k, \tilde{z}_k - z\rangle \quad \forall z \in \dom w.
\end{align*}

(b) This statement follows as an immediate consequence of (a) and (\ref{cond1}).

(c) 
Let $z^* \in (S+T)^{-1}(0)$. Then, there exists $a^*\in Z$ such that 
 $a^*  \in S(z^*)$ and $- a^*  \in T(z^*)$. In view of \eqref{subpro}, we can write 
 $r_k$ as $r_k=r_k^a+r_k^b$ where $r_k^a\in S(\tz_k)$ and $r_k^b\in T^{ \varepsilon_k}(\tz_k)$. Since $z^*, \tz_k \in Z$,
$a^* \in S(z^*)$ and $r_k^a\in S(\tz_k)$,  condition {\bf A1} implies  that  $\inner{r_k^a-a^*}{\tilde z_k-z^*}\geq \mu (dw)_{\tilde z_k}(z^*)$.
On the other hand,  since $-a^* \in T(z^*)$ and $r_k^b\in T^{ \varepsilon_k}(\tz_k)$,
  \eqref{eq:def.eps} implies  that $\inner{r_k^b+a^*}{\tilde z_k-z^*}\geq -\varepsilon_k$.
Hence, 
\begin{equation}\label{a390}
 \langle r_k, \tilde{z}_k - z^*\rangle + \varepsilon_k= \langle r_k^a-a^*, \tilde{z}_k - z^*\rangle +  \langle r_k^b+a^*, \tilde{z}_k - z^*\rangle + \varepsilon_k \geq \mu (dw)_{\tilde z_k}(z^*).
\end{equation}
Statement (b) with $z=z^*$ together wih the previous inequality then yield  (c).

(d) The first inequality of this statement follows directly from (c). Now, since 
$(dw)_{z_{k}}(\tilde z_{k})   \leq \sigma (dw)_{z_{k-1}}({\tz}_k)+\eta_{k-1}
 $ (see  \eqref{cond1}), the second inequality of this statement follows from the first one and
 the fact that $\sigma \in [0, 1)$.
\endproof

Under the assumption
that the sequence of stepsizes $\{\lambda_k\}$
is bounded away from zero, the following  result shows that the sequence $\{dw_{z_k}(z^*)+\eta_{k}\}$ converges geometrically to zero
for every solution $z^*$ of \eqref{eq:inc.p}. 

\begin{proposition}\label{lm:tech2}
Let $\mu$ be as in   {\bf A1}  and define
\begin{equation}\label{tau}
\underline{\tau}:=\min \left\{\frac{m}M\left(\frac{1}{1-\sigma}+\frac{1}{\mu \lambda}\right)^{-1} ,\tau\right\} \in (0,1).
\end{equation}
 Then, for  every $z^* \in (S+T)^{-1}(0)$ and $k > \ell \ge 0$, we have:
\begin{equation}\label{eq:094}
(dw)_{z_{k}}(z^*)+\eta_{k} \leq (1-\underline{\tau})^{k-\ell}\left[(dw)_{z_{\ell}}(z^*)+\eta_{\ell}\right], 
\end{equation}
\begin{equation}\label{eq:093}
\| \nabla(d{w})_{z^* }(\tilde z_k)\|^* \leq \frac{\sqrt{2}M}{\sqrt{m}}\left[1+ \frac{1}{\sqrt{1-\sigma}}\right](1-\underline{\tau})^{(k-\ell-1)/2}[(dw)_{z_{\ell}}(z^*)+\eta_{\ell}]^{1/2}. 
\end{equation}
\end{proposition}
\proof 
Let $  z^* \in (S+T)^{-1}(0)$  be given. 
It follows  from Lemma~\ref{lema_desigualdades}(c) and inequality \eqref{eq:56}    with 
 $t=\mu\lambda_k/(1-\sigma)$, $z=z_{k-1}$, $\tilde z=\tilde z_k$ and $z'=z^*$  that 
\begin{align*}
(dw)_{z_{k}}(z^*)+\eta_{k} &\leq \left(1 -\frac{m}M\left(\frac{1}{1-\sigma}+\frac{1}{\mu\lambda_k}\right)^{-1}\right)(dw)_{z_{k-1}}(z^*)+(1-\tau)\eta_{k-1}\\
&\leq \left(1 -\underline{\tau}\right)\left[(dw)_{z_{k-1}}(z^*)+\eta_{k-1}\right], \quad \forall k > 0
\end{align*}
where the second inequality is due to the fact that $\lambda_k\geq  \lambda$ for all $k$ and  the definition of $\underline{\tau}$ 
in~\eqref{tau}.
Clearly,  \eqref{eq:094} follows from last inequality.
Now, using \eqref{grad-d}, inequality \eqref{eq:789}  and the triangle inequality, we have
\begin{align*}
\| \nabla(d{w})_{z^* }(\tilde z_k)\|^* & \leq 
\| \nabla(d{w})_{z_{k-1} }( z^*)\|^*+\| \nabla(d{w})_{ z_{k-1}}(\tilde z_k)\|^*\\
&\leq  \frac{\sqrt{2}M}{\sqrt{m}}\left[ ( (d{w})_{z_{k-1} }( z^*))^{1/2}+((d{w})_{z_{k-1}}(\tilde z_k))^{1/2}\right]
\end{align*}
which, together with \eqref{eq:094} with $k=k-1$, the first inequality of Lemma~\ref{lema_desigualdades}(d) and
 the fact that $\eta_{k-1}\geq 0$, yield \eqref{eq:093}.
\endproof

The next result derives convergence rate bounds for the sequences $\{r_k\}$ and $\{\varepsilon_k\}$ generated by
an instance of Framework~1.
\begin{proposition} \label{pr:mmm} 
 Let $\underline{\tau}$ be as defined in \eqref{tau}. 
Then,
 for every $k\geq 1$, $r_k\in (S+T^{[\varepsilon_k]})(\tilde z_k)$
 and the convergence rate bounds hold
\begin{equation}\label{ad90}
 \|r_k\|^* \le   \frac{2\sqrt{2}M}{\lambda\sqrt{m}}(1-\underline{\tau})^{(k-1)/2}
 \sqrt{\underline{d}_0+\eta_0},
   \qquad
\varepsilon_k \leq   \frac{ 1 }{\lambda(1 - \sigma)}(1 - \underline{\tau})^{k-1}\left[\underline{d}_0+\eta_0\right]
\end{equation}
where   $\underline{d}_0:=\inf \{(d{w})_{z_0}(z): z \in (S+T)^{-1}(0)\}$.
\end{proposition}
\begin{proof}
 The first statement of the proposition  follows  from  \eqref{subpro}. 
Let  $z^* \in (S+T)^{-1}(0)$ be given. 
Using \eqref{grad-d},  \eqref{subpro}, $\lambda_k\geq  \lambda>0$,  the triangle inequality and inequality \eqref{eq:789}, we have
\begin{align*}
 \|r_k\|^*=  \frac{1}{ \lambda_k}\| \nabla (dw)_{z_k}(z_{k-1})\|^* &\leq  
 \frac{1}{ \lambda}\left [\| \nabla (dw)_{z_{k}}(z^{*})\|^*+\| \nabla (dw)_{z_{k-1}}(z^*)\|^*\right]\\
  &\leq   \frac{\sqrt{2}M}{\lambda\sqrt{m}}
\left[  ((dw)_{z_{k}}(z^{*}))^{1/2}+((dw)_{z_{k-1}}(z^*))^{1/2}\right]
\end{align*}
which combined with \eqref{eq:094} with $\ell=0$ yields 
\[
 \|r_k\|^*  \leq    \frac{\sqrt{2}M}{\lambda\sqrt{m}}\left[1+(1-\underline{\tau})^{1/2}\right](1-\underline{\tau})^{(k-1)/2}\left[(dw)_{z_{0}}(z^*)+\eta_{0}\right]^{1/2}.
\]
As $\underline{\tau}\in (0,1]$ (see \eqref{tau}) and $z^*$ is an arbitrary point in $(S+T)^{-1}(0)$, the bound on $r_k$ follows from the definition of $\underline{d}_0$.
Now,    
since $\lambda_k\geq  \lambda>0$, it follows from  \eqref{cond1}  that
\[ \lambda\varepsilon_k  \leq \lambda_k\varepsilon_k \leq \sigma (dw)_{z_{k-1}}({\tz}_k)+(1-\tau)\eta_{k-1}.
 \]
 On the other hand, Lemma~\ref{lema_desigualdades}(c) implies that
\[
 (1 - \sigma) (dw)_{z_{k-1}}({\tz}_k)\leq (dw)_{z_{k-1}}(z^*) +(1-\tau)\eta_{k-1}.
 \]
 Combining the last two  inequalities and the fact that $\sigma \in [0, 1)$, we obtain 
 \[ \lambda\varepsilon_k  \leq \frac{\sigma}{1 - \sigma}(dw)_{z_{k-1}}(z^*)+\frac{1-\tau}{1 - \sigma}\eta_{k-1}\leq \frac{1}{1 - \sigma}\left[(dw)_{z_{k-1}}(z^*)+(1-\tau)\eta_{k-1} \right],
 \]
which together with \eqref{eq:094} with $\ell=0$ and the fact that  $\tau >0$  imply that
\[ \varepsilon_k \leq  \frac{ (1 - \underline{\tau})^{k-1} }{ \lambda(1 - \sigma)} [(dw)_{z_{0}}(z^*)+\eta_0]. \]
Since $z^*$ is arbitrary point in $(S+T)^{-1}(0)$, the bound on $\varepsilon_k$ follows from the definition of $\underline{d}_0$.
\end{proof}

Proposition \ref{pr:mmm} gives convergence rate bounds on the last triple $(\tz_k,r_k,\varepsilon_k)$ generated by Framework 1.
The next result whose proof is given in Appendix~\ref{proofpropergodic} gives convergence rate bounds on the ergodic triple obtained by averaging the triples $(\tz_i,r_i,\varepsilon_i)$
from $i=\ell+1$ to $i=k$ where $k > \ell \ge 0$. More specifically, for $k> \ell \geq 0$,  define
$\Lambda_{\ell,k} := \sum_{i=\ell+1}^k \lambda_i$ and the $(\ell,k)$-ergodic triple $(\tilde z^a_{\ell,k},r^a_{\ell,k},\varepsilon^a_{\ell,k})$ as
\begin{equation}\label{SeqErg}
\tilde z^a_{\ell,k} := \frac{1}{\Lambda_{\ell,k}}\sum_{i=\ell+1}^k \lambda_i \tilde z_i, \quad
r^a_{\ell,k} := \frac{1}{\Lambda_{\ell,k}}\sum_{i=\ell+1}^k \lambda_i r_i, \quad
\varepsilon^a_{\ell,k} := \frac{1}{\Lambda_{\ell,k}} \sum_{i=\ell+1}^k \lambda_i \left( \varepsilon_i + \inner{r_i}{\tilde z_i -\tilde z^a_{l,k}} \right).
\end{equation}

\begin{proposition} \label{prop:ergodic}
 Let $\underline{\tau}$ be as defined in \eqref{tau}. 
Then,
 for every $k> \ell \geq 0$,
 \begin{equation}\label{qer10}
  \varepsilon^a_{\ell,k}\geq 0, \quad r^a_{\ell,k} \in \left(S+ T\right)^{[\varepsilon^a_{\ell,k}]}(\tilde z^a_{l,k})
  \end{equation}
 and the following convergence rate bounds hold:
\begin{equation}\label{ad901}
 \|r^a_{\ell,k}\|^* \le \frac{2\sqrt{2}M}{ \Lambda_{\ell,k}\sqrt{m}}(1 - \underline{\tau})^{\ell/2}\left[\underline{d}_0+\eta_0\right]^{1/2},
   \qquad
\varepsilon^a_{\ell,k} \leq \frac{9 M}{ m (1 - \sigma)\Lambda_{\ell,k} }
(1 - \underline{\tau})^{\ell}\left[\underline{d}_0+\eta_0\right]
\end{equation}
where   $\underline{d}_0:=\inf \{(d{w})_{z_0}(z): z \in (S+T)^{-1}(0)\}$. 
\end{proposition}

We end this section by making two remarks about Proposition \ref{prop:ergodic}. First, Proposition \ref{pr:mmm} implies  Proposition \ref{prop:ergodic}
when $\ell =k-1$ and yields a slightly better bound on $\varepsilon_k$.
Second, under the assumption that $\max\{ M, m^{-1} \}  = {\cal O}(1)$, Proposition \ref{prop:ergodic} with $\ell=0$ implies that
\[
\|r^a_{\ell,k}\|^* = {\cal O} \left(  \frac{1}{k \lambda} \left[\underline{d}_0+\eta_0\right]^{1/2} \right), \quad
\varepsilon^a_{\ell,k} = {\cal O} \left(  \frac{1}{k \lambda} \left[\underline{d}_0+\eta_0\right] \right)
\]
and with  $\ell=\lceil k/2\rceil$ implies that
\[
\|r^a_{\ell,k}\|^* = {\cal O} \left(  \frac{1}{k \lambda} (1-\underline{\tau})^{k/4} \left[\underline{d}_0+\eta_0\right]^{1/2} \right), \quad
\varepsilon^a_{\ell,k} = {\cal O} \left(  \frac{1}{k \lambda}  (1-\underline{\tau})^{k/2} \left[\underline{d}_0+\eta_0\right] \right).
\]
Hence, the $(\lceil k/2\rceil,k)$-ergodic triple has the property that
$k(r^a_{\ell,k},\varepsilon^a_{\ell,k})$ converges to $0$ geometrically.

\section{{A regularized  NE-HPE 
framework for solving MIPs}}
\label{sec:imp.hpe}

This section describes and establishes the pointwise iteration-complexity of a regularized NE-HPE framework for solving MIPs
which, specialized to the case of the Euclidean Bregman distance and error sequence $\{\eta_k\}$ identically zero,
reduces to the regularized HPE framework of \cite{Maicon}.
The latter framework has been shown in \cite{Maicon} to have
better iteration-complexity than the one for the usual HPE framework derived in \cite{monteiro2010complexity}.
Moreover, the derived pointwise iteration-complexity bound for the case of a general Bregman distance is, up to a logarithm factor,
the same as the ergodic iteration-complexity bound for the standard NE-HPE method obtained
in \cite{Oliver}.

Our problem  of interest in this section is the MIP
\begin{align} \label{eq:in}
 0\in T(z)
\end{align}
where $T:\Z\tos \Z$ is a maximal monotone operator such that
 the solution set $T^{-1}(0)$ of~\eqref{eq:in} is nonempty. 


We also assume in this section that, for some positive scalars $m$ and $M$,
 $w:\Z \to [-\infty,\infty]$ is an $(m,M)$-regular distance generating function 
with respect to $(Z,\|\cdot\|)$ where $Z \subset \inte(\dom w)$ is a convex set such that $\Dom(T)\subset Z$ and 
$\|\cdot\|$ is a seminorm in~$\Z$.
The regularized NE-HPE framework solves \eqref{eq:in} based on the idea of solving the regularized MIP
\begin{align} \label{eq:inre}
 0\in T(z)+\mu \nabla (d{w})_{z_0}(z)
\end{align}
for a fixed  $z_0\in Z$   and a sufficiently small $\mu>0$.
Hence, we also assume that  the solution set of~\eqref{eq:inre}
\begin{equation}\label{defZu}
Z^*_{\mu}:=\{z\in Z: \; 0 \; \in \; T(z)+\mu \nabla(d{w})_{z_0}(z)\}
\end{equation}
  is nonempty for every $\mu>0$.  We remark that if 
$w$ is strongly convex on $Z$, then Proposition~\ref{existence} with $T=(1/\mu)T$ and $\hat z=z_0$
implies that the latter assumption holds.

Note that \eqref{eq:inre} is a special case of~\eqref{eq:inc.p}
with $S(\cdot)=\mu \nabla (d{w})_{z_0}(\cdot)$, and from  the above assumptions the operators $S$ and $T$ satisfy   {\bf A0} and {\bf A2}.
Moreover,  this operator $S$ together with $w$ and $Z$  satisfies {\bf A1}.  Indeed, 
 using the definition of $S$ and  \eqref{grad-d}, we conclude that for every $z, z' \in   Z$, 
\begin{align*}
 \inner{S(z)-S(z')}{z-z'} &= \mu \inner{\nabla (d{w})_{z_0}(z)- \nabla (d{w})_{z_0}(z')}{z-z'}\\
 & = \mu \inner{\nabla (d{w})_{z'}(z)}{z-z'}=  \mu [(dw)_{z}(z')+(dw)_{z'}(z) ]
\end{align*}
 where the last equality is due to \eqref{equacao_taylor} with $v=z'$.
 Hence, we can use any instance of  Framework~1 with $S(\cdot)=\mu \nabla (d{w})_{z_0}(\cdot)$ to approximately solve the regularized inclusion  \eqref{eq:inre}, 
 and hence \eqref{eq:in} when $\mu>0$ is sufficiently small.

For every  $\mu>0$, define
\begin{align}  \label{eq:dmuxmu}
 d_0:=\inf_{z \in T^{-1}(0)} (d{w})_{z_0}(z), \quad   d_\mu:=\inf_{z \in Z^*_{\mu} } (d{w})_{z_0}(z).
\end{align}
The following simple result  establishes a crucial relationship between $d_0$ and $d_\mu$. 
\begin{lemma}\label{lm:dz}
For any $\mu>0$ and $z^*_\mu \in Z^*_\mu$, there holds  $(dw)_{z_0}(z^*_\mu)\leq d_0$. As a consequence, $d_\mu\leq d_0$.
\end{lemma}
\begin{proof}
Let $\mu>0$ and  $z^*_\mu \in Z^*_{\mu}$ be given. Clearly,  $-\mu \nabla(d{w})_{z_0}(z^*_\mu)\in T(z^*_\mu)$.
Hence, monotonicity of $T$ implies that any $z^* \in T ^{-1}(0)$ satisfies
%
 $\inner{ \nabla(d{w})_{z_0}(z^*_\mu))}{z^*-z^*_\mu }\geq 0 $.
%
The latter conclusion and relation \eqref{equacao_taylor} with $v=z_0$, $z'=z^*$ and $z=z_\mu^*$ then imply that
\[
(d{w})_{z_0}(z^*) - (d{w})_{z_0}(z_\mu^*)= \langle \nabla (d{w})_{z_0}(z^*_\mu), z^*-z^*_\mu\rangle + (d{w})_{z^*_\mu}(z^*)\geq 0.
\]
As $z^*\in T ^{-1}(0) $ is arbitrary, the first part of the lemma  follows from the definition of $d_0$. The second part of the lemma follows from the first one and the definition 
of $d_\mu$.
\end{proof}

 Note that, in view of Proposition~\ref{pr:mmm},  any instance of Framework~1 applied to \eqref{eq:inre}  generates a triple $(\tilde z_k,r_k,\varepsilon_k)$ such that
 \[
 \tilde r_k := r_k -\mu \nabla (d{w})_{z_0}(\tz_k) \in T^{[\varepsilon_k]}(\tilde z_k) 
 \]
and the residual pair $(r_k,\varepsilon_k)$ satisfies \eqref{ad90} with $\underline{d}_0=d_{\mu}$, and hence converges to $0$.
Even though the sequence $\tilde r_k$ does not necessarily converge to $0$, it can be made sufficiently
small, i.e., $\|\tilde r_k\|^* \le \rho$ for some tolerance $\rho>0$, by choosing $\mu=\rho/{\cal O}(d_0)$.
Indeed, we will show later that there exists
$\D_0 = {\cal O}(d_0)$ such that $\|\nabla (d{w})_{z_0}(\tz_k)\|^*\le \D_0$ for every $k$.
Hence, choosing $\mu = \rho/\D$ for some $\D \ge 2 \D_0$
and computing a triple $(\tilde z_k,r_k,\varepsilon_k)$ such that $\|r_k\|^* \le \rho/2$ guarantees that
\[
\|\tilde r_k\|^* \le \|r_k\|^* +\mu \| \nabla (d{w})_{z_0}(\tz_k)\|^* \le \frac{\rho}2 + \mu \D_0 = \rho \left( \frac{1}2 +  \frac{\D_0}\D \right) \le  \rho,
\]
and hence that $\tilde r_k$ is a sufficiently small residual for \eqref{eq:in}. Moreover, Proposition \ref{pr:mmm} implies
the iteration-complexity of the proposed scheme
increases as  $\D$ does, or equivalently, $\mu$ decreases. As a result, the best strategy is to choose a scalar $\D \ge 2 \D_0$ 
such that $\D={\cal O}(\D_0)$.
A technical difficulty of the proposed scheme is that $\D_0$ can not be explicitly computed since it depends on $d_0$ which is generally not known.

Our first framework below is essentially Framework 1 applied to \eqref{eq:inre} with an arbitrary guess of $\D$, and hence of $\mu=\rho/\D$.
In view of the above discussion, it is guaranteed to work only for large values of $\D$, i.e., when
$\D \ge 2 \D_0$. 
Subsequently, we present a dynamic scheme (see Framework~3) which successively calls Framework 2 for a sequence of increasing values of $\D$.
It is shown in Theorem \ref{th:main} that the latter scheme has the same iteration-complexity as
the best one for Framework 2 (i.e., the one obtained under the hypothetical assumption that a  scalar $\D \ge 2 \D_0$ and $\D={\cal O}(\D_0)$  is known).

\vgap
\vgap

\noindent
\fbox{
\begin{minipage}[h]{6.4 in}
{\bf Framework 2}
{(A static regularized NE-HPE framework for solving \eqref{eq:in}).} 
\begin{itemize}
\item[] \mbox{Input:} $(z_0, \eta_0,\mathcal{D})\in Z\times \R_{+}\times\R_{++}$ and $(\sigma,\tau,\lambda,\rho,
\varepsilon)\in  [0,1)\times  (0,1) \times \R_{++}
\times \R_{++}\times \R_{++}$; 
\item[(0)] set $ \mu= \rho/\mathcal{D}$ and $k=1$;
\item[(1)] choose $\lambda_k \ge \lambda$ and find 
$(z_k,\tz_k,\varepsilon_k,\eta_k)\in Z\times Z\times \R_+\times \R_+$  
       such that

       \begin{equation}\label{eq:ec.2}
  r_k:= \frac{1}{\lambda_k} \nabla (d{w})_{z_k}(z_{k-1})  \in \left(\mu \nabla (d{w})_{z_0}(\tz_k)+T^{[\varepsilon_k]}(\tz_k)\right), 
  \end{equation}
      \begin{equation}	\label{eq:es.2}
 (d{w})_{z_k}({\tz}_k) + \lambda_k\varepsilon_k+\eta_k \leq \sigma (d{w})_{z_{k-1}} (\tz_{k})+(1-\tau)\eta_{k-1}; 
    \end{equation}
       
\item[(2)] if $\norm{r_k}^*\leq \rho/2$
and $\varepsilon_k\leq \varepsilon$,
then stop and output
$(\tilde{z}_k,\tilde r_k,\varepsilon_k)$ where
\[
\tilde r_k = r_k - \mu \nabla (d{w})_{z_0}(\tz_k);
\]
otherwise, set  $k\leftarrow k+1$ and go to step 1. 
\end{itemize}
\noindent
{\bf end}
\end{minipage}
}
\vgap
\vgap

We now make two remarks about Framework~2.
First, as mentioned above,
it is the special case of Framework~1 in which $S(\cdot)=\mu \nabla (d{w})_{z_0}(\cdot)$,
and hence solves MIP \eqref{eq:inre}.
Second, since Section~\ref{sec:smhpe} only deals with
convergence rate bounds, a stopping criterion was not added to Framework~1.
In contrast, Framework~2 incorporates
a stopping criterion (see step 2 above) based on which  its iteration-complexity bound is obtained.
Clearly, \eqref{eq:ec.2} together with the termination criteria  $\|r_k\|^* \le \rho/2$ and
$\varepsilon_k \le \varepsilon$ provides a certificate of the quality of $\tilde z_k$ as an approximate
solution of the regularized MIP \eqref{eq:inre}.

The next result establishes the pointwise iteration-complexity of Framework 2
and shows that any instance of Framework 2 also solves the original MIP \eqref{eq:in} when $\D$ is sufficiently large.

\begin{theorem}\label{cr:c.alg2} 
The following statements hold:
\begin{itemize}
\item[(a)]
Framework~2 terminates in at most
\begin{equation}\label{eq:it.cc}
\max\left\{
\frac{M}{m}\left(\frac{\mathcal{D}}{ \lambda\rho}
 +\dfrac{1}{1-\sigma}\right),\frac{1}{\tau}\right\}
\left[ 2 + 
\max\left\{\log^+\left(\dfrac{32M^2(d_0+\eta_0)}
{ (\lambda\rho)^2m}\right),
\log^+\left(\dfrac{d_0+\eta_0}
{ \lambda(1-\sigma)\varepsilon}\right)
\right\} \right]
\end{equation}
iterations with a triple  $(\tilde{z}_k,\tilde r_k,\varepsilon_k)$  satisfying the following conditions 
\begin{align*}
\tilde r_k \in T^{[\varepsilon_k]}(\tilde{z}_k), \quad  \norm{\tilde r_k + \mu\nabla(d{w})_{z_0}(\tilde z_k)}^*\leq \rho/2, \quad 
\varepsilon_k\leq \varepsilon;
\end{align*}
\item[(b)] for every $k\ge 1$,  
\[
\| \nabla(d{w})_{z_0}(\tilde z_k)\|^*
\le \D_0:= \frac{\sqrt{2}M}{\sqrt{m}}\left[{2}+\dfrac{1}{\sqrt{1-\sigma}}\right]\left(d_{0}+\eta_{0}\right)^{1/2};
\]
\item[(c)]
 if $\mathcal{D} \ge 2\D_0$, 
 then $\norm{\tilde r_k}^*\leq \rho.$
\end{itemize}
\end{theorem}
\begin{proof}
(a) Assume
that  Framework~2 has not terminated at the $k$-th
iteration.
Then, either
$\norm{r_k}^*>\rho/2$ or $\varepsilon_k>\varepsilon$.
Assume first that $\norm{r_k}^*>\rho/2$.
Since Framework 2 is a special case of Framework~1
applied to MIP \eqref{eq:inre}
with $S(z)=\mu \nabla(d{w})_{z_0}(z)$,
the latter assumption 
and Corollary~\ref{pr:mmm} imply that
\[
 \frac{\rho}{2}<\|r_k\|^*\leq  \frac{2\sqrt{2}M}{\lambda\sqrt{m}}(1-\underline{\tau})^{(k-1)/2}
 (d_{\mu}+\eta_0)^{1/2}\leq  \frac{2\sqrt{2}M}{\lambda\sqrt{m}}(1-\underline{\tau})^{(k-1)/2}
 (d_{0}+\eta_0)^{1/2}
 \]
 where the last inequality is due to Lemma~\ref{lm:dz}.
Rearranging the last inequality, taking logarithms of both sides of the resulting  inequality and using
the fact that
$\log(1-\underline{\tau})\leq -\underline{\tau}$,  we conclude that
\begin{equation*}
 k< 1+\underline{\tau}^{-1}\log
\left(\dfrac{32M^2(d_0+\eta_0)}
{( \lambda\rho)^2m}\right).
\end{equation*}
If, on the other hand, $\varepsilon_k>\varepsilon$,
we conclude by using a similar reasoning that
\begin{equation*}
 k< 1+\underline{\tau}^{-1}\log
\left(\dfrac{d_0+\eta_0}
{(1-\sigma) \lambda\varepsilon}\right).
\end{equation*}
The complexity bound in (a) now follows from the above observations, the definition of $\underline{\tau}$ in \eqref{tau},
and the fact that $\mu=\rho/\mathcal{D}$.

(b)  By first considering \eqref{eq:093} with $\ell=k-1$ and then using \eqref{eq:094}, we have
\[
\| \nabla(d{w})_{z_\mu^* }(\tilde z_k)\|^* \leq  \frac{\sqrt{2}M}{\sqrt{m}}\left(1+ \frac{1}{\sqrt{1-\sigma}}\right)((dw)_{z_{0}}(z_\mu^*)+\eta_{0})^{1/2}
\]
for an arbitrary point  $ z^*_\mu \in Z^*_{\mu}$.
Hence, using \eqref{grad-d}, \eqref{eq:789} and the triangle inequality, we obtain
\begin{align*}
 \| \nabla(d{w})_{z_0}(\tilde z_k)\|^* &\leq \|\nabla d w_{ z_0}{(z_\mu^*) } \|^*+ \| \nabla(d{w})_{z_\mu^* }(\tilde z_k)\|^* \\
  &\leq  \frac{\sqrt{2}M}{\sqrt{m}}\left[( (d{w})_{z_0}(z_\mu^*))^{1/2}+\left(1+\frac{1}{\sqrt{1-\sigma}}\right)( (d{w})_{z_0}(z_\mu^*)+\eta_0)^{1/2}\right]\\
  &\leq  \frac{\sqrt{2}M}{\sqrt{m}}\left[2+ \frac{1}{\sqrt{1-\sigma}}\right]\left((d{w})_{z_0}(z_\mu^*)+\eta_{0}\right)^{1/2}
 \end{align*}
 which implies  the conclusion of (b),  in view of  Lemma~\ref{lm:dz} and the definition of $\D_0$.

(c) This statement follows  immediately from (a) and (b) (see the paragraph following Lemma~\ref{lm:dz}).

\end{proof}

We now make some remarks about Theorem~\ref{cr:c.alg2}.
First, if $(1-\sigma)^{-1}$ and $\tau^{-1}$  are    $\mathcal{O}(1)$,   and an input $\D$ for Framework~2 satisfying 
 $2\D_0\leq {\cal D}= {\cal O}(\D_0)$ 
is known,
then the complexity bound \eqref{eq:it.cc}  becomes
\begin{align}
\label{eq:it.c3}
 \mathcal{O}\left(\frac{M}{m}\left(
\dfrac{M(d_0+\eta_0)^{1/2}}
{  \lambda\rho\sqrt{m}}
+1\right)
\left[ 1 + 
\max\left\{\log^+\left(\dfrac{M^2(d_0+\eta_0)}
{ (\lambda\rho)^2m}\right),
\log^+\left(\dfrac{d_0+\eta_0}
{ \lambda\varepsilon}\right)
\right\} \right]\right),
\end{align}
in view of the definition of $\D_0$ in Theorem \ref{cr:c.alg2}(b).
Second, in general an upper bound 
${\cal D}$ as in the first remark is not 
known and, in such a case, bound~\eqref{eq:it.cc} 
can be much worse than the one above, e.g., when ${\cal D} \gg 2\D_0$.

 We now consider the case where an upper bound ${\cal D} \ge 2\D_0$  such that
${\cal D}= {\cal O}(\D_0)$ is not known and describe a scheme
based on Framework 2 whose iteration-complexity bound is equal to  \eqref{eq:it.c3}.

\vgap
\vgap

\noindent
\fbox{
\begin{minipage}[h]{6.4 in}
{\bf Framework 3} (A dynamic regularized NE-HPE framework for solving \eqref{eq:in}).
\begin{itemize}
\item[(0)] Let $z_0\in \Z$, $\eta_0\geq0$, $\sigma\in [0,1)$, $\tau \in (0,1)$,  $\lambda>0$
and a tolerance 
pair $( \rho, \varepsilon) \in 
\R_{++} \times \R_{++}$ be given and
set $\mathcal{D}= \lambda\rho;$
\item[(1)]  call Framework~2 with input
$(z_0,\eta_0,\mathcal{D})$ and $(\sigma, \tau,\lambda, \rho,\varepsilon)$
to obtain $(\tilde{z},\tilde r,\tilde{\varepsilon})$  as output;
%
\item[(2)]  
if  $\norm{\tilde r}^*\leq \rho$
then stop and output $(\tilde{z},\tilde r,\tilde{\varepsilon})$;
else, set $\mathcal{D} \leftarrow 2\mathcal{D}$
and go to step~1. 
\end{itemize}
\noindent
{\bf end}
\end{minipage}
}
\vgap
\vgap
\vgap

Each iteration of Framework~3 (referred to as an outer iteration) invokes Framework 2, which performs a certain number of iterations (called  inner iterations)
which in turn is bounded by \eqref{eq:it.cc}. The following result gives the overall inner
iteration-complexity of Framework~3.
\begin{theorem}\label{th:main}
 {\rm Framework~3} with input 
$(z_0, \eta_0)\in Z\times \R_{+}$ and $(\sigma,\tau, \lambda, \rho, \varepsilon)\in  [0,1)\times (0,1)\times \R_{++}\times \R_{++}
\times \R_{++}$ such that  $(1-\sigma)^{-1}$ and $1/\tau$ are    $\mathcal{O}(1)$ 
  finds a  triple
$(\tilde{z} ,\tilde r,\tilde{\varepsilon})$ satisfying
\[
\tilde r \in T^{[\tilde \varepsilon]}(\tilde{z}),\quad \norm{\tilde r}^*\leq \rho,
\quad \tilde{\varepsilon}\leq \varepsilon
\]
by performing a total number of inner iterations bounded by \eqref{eq:it.c3}.
\end{theorem} 
\begin{proof}
Note that at 
the $k$-th outer iteration of Framework~3,  
we have $\mathcal{D}= 2^{k-1} \lambda \rho$. Hence, taking $\D_0$ as in Theorem~\ref{cr:c.alg2}(b), it follows from
Theorem \ref{cr:c.alg2}(c) that
Framework~3 terminates in at most $K$ outer iterations where $K$ is the smallest integer $k \ge 1$ satisfying 
$2^{k-1}\lambda \rho \ge 2\D_0$. Thus, 
\begin{align*}
 K=1+\left\lceil\log^+\left(
{2\D_0}/(\lambda\rho)\right)\right\rceil.
\end{align*}
%
%
%
In order to simplify the calculations, let us denote
\begin{align}\label{beta_1}
\beta_1 &:=2 + 
\max\left\{\log^+\left(\dfrac{32M^2(d_0+\eta_0)}
{ (\lambda\rho)^2m}\right),
\log^+\left(\dfrac{d_0+\eta_0}
{ \lambda(1-\sigma)\varepsilon}\right)
\right\}. 
\end{align}
In view of Theorem \ref{cr:c.alg2}(a) and \eqref{beta_1}, we see that 
the overall  number of inner iterations  is bounded by
\[
\widetilde K:=
\beta_1\sum_{k=1}^{K}\,\max\left\{\frac{M}{m}\left({2^{k-1}}+\frac{1}{1-\sigma}\right),\frac{1}{\tau}\right\}
\leq \beta_1\left[\frac{M}{m}\left({2^K-1}+\frac{K}{1-\sigma}\right)+\frac{K}{\tau}\right],
\]
and hence
\begin{align}\label{eq:ktilde}
\widetilde K&\leq \frac{M\beta_1}{m}\left[{1}+\frac{1}{1-\sigma}+\frac{1}{\tau} \right]2^K.
\end{align}
To end the proof, it suffices 
to show that $\widetilde K$ is bounded by \eqref{eq:it.c3}.
If $K=1$, then  \eqref{beta_1} combined with \eqref{eq:ktilde} and the fact that $(1-\sigma)^{-1}$ and  $1/\tau$  are    $\mathcal{O}(1)$
shows  that \eqref{eq:it.c3} trivially holds. Assume now that $K>1$ and note that $k:=K-1$ violates
the inequality $2^{k-1}\lambda\rho
\ge 2\D_0$, and hence 
that $2^{K}<8\D_0/(\lambda\rho)$.
The latter estimate 
combined with inequality~\eqref{eq:ktilde}  implies that
$$\widetilde {K}<\frac{8M\beta_1}{m}\left[ {1}+\frac{1}{1-\sigma}+\frac{1}{\tau} \right] \frac{\D_0}{\lambda\rho}$$ 
which together with  \eqref{beta_1}, and the fact that $(1-\sigma)^{-1}$ and  $1/\tau$  are    $\mathcal{O}(1)$ and
$\D_0= {\cal O}(M(d_0+\eta_0)^{1/2}/\sqrt{m})$, imply that $\widetilde K$ 
is bounded by \eqref{eq:it.c3}.
\end{proof}

Note that if $\varepsilon_k=0$ for every $k$, then the complexity bound \eqref{eq:it.c3}  becomes independent of the tolerance  $\varepsilon$, namely, it  
 reduces  to
\begin{align}
\label{eq:itfram2eps0}
 \mathcal{O}\left(\frac{M}{m}\left(\frac{M(d_0+\eta_0)^{1/2}}{\lambda\rho\sqrt{m}}+1\right)\left[1+
\log^+\left(
\dfrac{M^2(d_0+\eta_0)}
{  (\lambda\rho)^2m}\right)
\right]\right).
\end{align}

We now make two remarks about Frameworks 2 and 3.
First, Framework 3 assumes that every call to Framework 2 in step 1 uses the same distance generating function $w$.
However, it is possible to use a different  distance generating
function in each call to Framework 2 and still preserve the iteration-complexity bound
obtained in Theorem \ref{th:main}.
Indeed, if for a given $d_0 >0$, step 1 of Framework 3  calls Framework 2 with a probably different
distance generating function $w$ which remains the same throughout the same call, is
$(m,M)$-regular with respect to $(Z,\|\cdot\|)$ as in Definition \ref{def:assu},
and satisfies
\[
 \inf_{z \in T^{-1}(0)} (d{w})_{z_0}(z) \le d_0,
\]
then the resulting variant of Framework 3 will have the same iteration-complexity bound \eqref{eq:it.c3}.
Second, the termination criterion used by Framework 2 (and consequently Framework 3)
is based on the last generated triple $(\tx_k,\tilde r_k,\varepsilon_k)$. Instead of it, Framework 2 can also
use an ergodic stopping criterion, namely: it terminates
 when $\|r_{\ell,k}^a\|^* \le \rho/2$ and
$\varepsilon_{\ell,k}^a\leq \varepsilon$ and then output the triple $(\tz,\tilde r,\tilde \varepsilon) = (\tz_{\ell,k}^a, \tilde r_{\ell,k}^a,\varepsilon_{\ell,k}^a)$
where
\[
\tilde r_{k}^a := r_{k}^a - \mu \sum_{i=1}^k\frac{\lambda_i}{\Lambda_k}\nabla (d{w})_{z_0}(\tilde z_{i})
\]
It can be shown with the aid of Proposition \ref{prop:ergodic} that when $\ell=0$ and ${\cal D}$ is properly initialized,
the iteration complexity of the modified (ergodic) Framework 3 is ${\cal O}(\max \{\rho^{-1},\varepsilon^{-1}\})$.
We omit the details for the sake of shortness but note that the bound is essentially the same as the ergodic iteration complexity
bound for the standard HPE framework 
obtained in \cite{monteiro2010iteration}.


 \section{A regularized  ADMM class }
\label{sec:amal}

The goal of this section is to present a regularized  ADMM class
for  solving  linearly constrained convex programming problems
 which has
a better pointwise iteration-complexity than the standard ADMM.
It contains two subsections. The first one describes our setting, 
our assumptions, the regularized ADMM class  and its improved pointwise iteration-complexity bound. 
The second one is dedicated to the proof of the main result  stated in the first subsection.

 \subsection{A regularized ADMM class and its pointwise iteration-complexity}\label{subsec:Admm1}

In this subsection,  let $\Sf$, $\Y$ and $\X$ be finite-dimensional real vector spaces  with  inner products
 denoted by $\inner{\cdot}{\cdot}_\mathcal{S}$, $\inner{\cdot}{\cdot}_\Y$ and
$\inner{\cdot}{\cdot}_\X$, respectively. Let us also consider the norm in $\X$ given by   $\norm{\cdot}_\X:= \inner{\cdot}{\cdot}_\X^{1/2}$,
and the seminorms in
$\mathcal{S}$ and $\Y$ defined by 
\begin{equation}\label{seminorm12}
\|\cdot\|_{\mathcal{S},H}:=\inner{H(\cdot)}{\cdot}^{1/2},  \quad \|\cdot\|_{\Y,G}:=\inner{G(\cdot)}{\cdot}^{1/2},
\end{equation}
respectively, where  $H:\mathcal{S} \to \mathcal{S}$ and $G:\Y \to \Y$ are self-adjoint positive semidefinite linear operators.

Our problem of interest is
\begin{equation} \label{optl}
\inf \{ f(y) + g(s) : C y + D s = c \}
\end{equation}
where
$c \in \X$,
$C: \Y \to \X$ and $D: \Sf \to \X$ are linear operators,
and $f: \Y \to (-\infty,\infty]$ and $g:\Sf \to (-\infty,\infty]$ are proper
closed convex functions. 
The following assumptions are made throughout this section:
\begin{itemize}
  \item[\bf B1)] the
  problem \eqref{optl} has an optimal solution $(s^*,y^*)$ and an associated Lagrange multiplier $x^*$, or equivalently,
   the monotone inclusion
\begin{equation} \label{FAB}
0\in T(s,y,x) := \left[ \begin{array}{c}  \partial g(s)- D^*{x}\\ \partial f(y)- C^* {x} \\ Cy+Ds-c
\end{array} \right] 
\end{equation}
has a solution $(s^*,y^*,x^*)$;
  \item[\bf B2)]  there exists  $x\in \X$
  such that $(C^*x,D^*x) \in \ri(\dom f^*) \times \ri(\dom g^*)$.
\end{itemize}

We now make two
remarks about the above assumptions. First, it follows from the last conclusion of
Proposition \ref{lem:auxSol} in Appendix~\ref{basiresul} that,
 if  the solution set of  \eqref{optl} is  nonempty and bounded, then {\bf B2} holds.
Second, by Proposition \ref{lem:auxSol}(a), if the infimum in \eqref{optl} is finite and {\bf B2} holds,
 then \eqref{optl} has an optimal solution. Hence,
{\bf B2} together with the Slater condition that there exists a feasible pair
$(s,y)$ for \eqref{optl} such that $(s,y) \in \ri(\dom g) \times \ri (\dom f) $ imply that
{\bf B1} holds (see Proposition \ref{lem:auxSol}(c)).

We are ready to state the regularized ADMM class for solving \eqref{optl}.
\\[3mm]
\noindent
{\bf Dynamic regularized alternating direction method of multipliers (DR-ADMM):}
\begin{itemize}
\item[(0)] Let $(s_0,y_0,x_0) \in \mathcal{S}\times \Y\times  \X$, positive scalars $\beta$ and $ \theta$, a tolerance $\rho>0$,
two self-adjoint positive semidefinite linear operators  $H: \mathcal{S} \to \mathcal{S}$ and  $G: \Y \to \Y$ which define
the seminorms \eqref{seminorm12} be given, and set $\mathcal{D}=\rho$;
\item[(1)]   set $\mu= \rho/\mathcal{D}$, 
$\beta_1=\beta\theta/(\theta+\mu)$ and $\beta_2=\beta(1+\mu)$ and $k=1$, and go to (a);
\begin{itemize}
\item[(a)] set
\begin{equation}
\hat{s}_{k-1}=\frac{ s_{k-1}+\mu s_0}{1+\mu}, \quad
\hat{y}_{k-1}=\frac{y_{k-1}+\mu y_0}{1+\mu}, \quad  \hat{x}_{k-1}=\frac{\theta x_{k-1}+\mu x_0}{\theta+\mu}
\label{defyhk}
\end{equation}
and compute an optimal solution $s_k \in \Sf$  of the subproblem
\begin{equation} \label{def:tsk-admm}
\min_{s \in \mathcal{S}} \left \{ g(s) - \inner{ D^*\hat{x}_{k-1}}{s}_\mathcal{S} +
\frac{\beta_1}{2} \| C y_{k-1} + D s - c \|^2_\X+\frac{1+\mu}{2}\|s-\hat s_{k-1}\|_{\mathcal{S},H}^2 \right\};
\end{equation}

\item[(b)]  set $\tilde{x}_k$ 
and $u_k$ as
\begin{align}
\tilde{x}_k&=\hat{x}_{k-1}-\beta_1(Cy_{k-1}+Ds_k-c)\label{defxtk}\\[1mm]
u_k&=\tilde{x}_{k}+\beta_2 (C\hat{y}_{k-1}+Ds_k-c) \label{defuk}
\end{align}
and compute an optimal solution $y_k\in \Y$ of the subproblem
\begin{equation} \label{def:tyk-admm}
\min_{y \in \Y} \left \{ f(y) - \inner{ C^*u_{k}}{y}_\Y +
\frac{\beta_2}{2} \| C y + D s_k - c \|^2_\X +\frac{1+\mu}{2}\|y-\hat y_{k-1}\|_{\Y,G}^2\right\};
\end{equation}

\item[(c)] update $x_k$  as 
\begin{equation}\label{admm:eqxk}
x_k = x_{k-1}-\theta\beta\left[Cy_k+Ds_k-c+\frac{\mu}{\theta\beta}(\tx_k-x_0)\right]
\end{equation}
and compute
\begin{equation}\label{defr}
b_k:=s_{k-1}-s_k, \quad a_k:=y_{k-1}-y_k, \quad q_k := \beta C(y_{k-1}-y_k), \quad  p_k :=  \frac{1}{\beta\theta}(x_{k-1}-x_k);
  \end{equation}

\item[(d)] if 
\begin{equation}\label{innerstop}
\left(\|b_k\|_{\mathcal{S},H}^2+\|a_k\|_{\Y,G}^2+ \frac{1}{\beta} \norm{q_k}_\X^2+\beta\theta\|p_k\|_\X^2\right)^{1/2} \leq \rho/2,
\end{equation}
then set $(s,y,\tilde{x})=(s_k,y_k,\tilde{x}_k)$, compute $(\tilde b, \tilde a,\tilde q, \tilde p)$ as
\begin{equation}\label{eq:rxry}
\tilde b:= b_k-{\mu} (s_k-s_0), \; \tilde a:= a_k-{\mu} (y_k-y_0), \; \tilde q:=  q_k-\mu \beta C(y_k-y_0),  \;  \tilde p:= p_k-\frac{\mu}{\beta\theta} (\tilde{x}_k-x_0)
\end{equation}
and go to (2); else set $k \leftarrow k+1$ and go to~(a);  
\end{itemize}

\item[(2)]   if 
\begin{equation}\label{Scri2}
\left(\|\tilde b\|_{\mathcal{S},H}^2+\|\tilde a\|_{\Y,G}^2+\frac{1}{\beta}\norm{\tilde q}_\X^2+\beta\theta\norm{\tilde p}_\X^2\right)^{1/2} \le \rho,
\end{equation}
then  stop and output $(s,y,\tilde{x},\tilde b, \tilde a,\tilde q, \tilde p)$; otherwise,
set $\mathcal{D}\leftarrow2\mathcal{D}$ and go to step~1.
\end{itemize}
\noindent
{\bf end}
\\[2mm]

We now make some remarks about the DR-ADMM.
First, assumption {\bf B2} together with Corollary \ref{cr:auxSol} imply that both subproblems
\eqref{def:tsk-admm} and \eqref{def:tyk-admm} have optimal solutions and hence DR-ADMM is well-defined.
Second, loop (a)-(d) with $\mu=0$ is exactly the ADMM  class \eqref{ADMMclass} with penalty parameter $\beta$ and relaxation stepsize $\theta$
(see for example \cite{glowinski1984}) since in this situation we have
$\beta_1=\beta_2=\beta$, $\hat{x}_{k-1}=u_k=x_{k-1}$ and $\hat{y}_{k-1}=y_{k-1}$.
However, it should be emphasized that DR-ADMM requires $\mu>0$. Hence, it does not belong to the ADMM class  \eqref{ADMMclass} and the results obtained for the DR-ADMM in this section do not apply to the latter class.
Third, DR-ADMM should essentially be viewed as a regularized variant of ADMM which dynamically adjusts the regularization parameter $\mu>0$,
or equivalently, the parameter $\D>0$ (as in Section~\ref{sec:imp.hpe}). 
Indeed, it will be shown later on (see Lemmas \ref{lemdeltaxy}, \ref{th:hpeheta<1} and \ref{th:hpeheta>1})
that DR-ADMM is a special instance of  Framework 3
in which $\varepsilon_k=0$ for all $k \ge 1$.
More sepecifically, steps 0, 1 and 2 of DR-ADMM correspond exactly to steps 0, 1 and 2 of Framework 3, respectively.
A single execution of steps 1 and 2 is referred to as an outer iteration of DR-ADMM.
A single execution of steps (a)-(d) is referred to as an inner iteration of DR-ADMM which, in the context of Framework 3,
corresponds to an iteration of Framework 2 (see step 1 of Framework 3). The cycle of inner iterations consisting of (a)-(d)
corresponds to the implementation of a special instance of Framework 2 in which 
$\varepsilon_k=0$ for all $k \ge 1$. Moreover, the two residuals $r_k$ and $\tilde r$ computed at the end of steps 1 and 2 of Framework 2,
respectively, correspond in the context of the DR-ADMM to the pairs $(Hb_k,Ga_k+C^*q_k,p_k)$ and 
$(H\tilde b,G\tilde a+C^*\tilde q,\tilde p)$, respectively
(see Lemma \ref{lemdeltaxy}).

We also make a remark about the subproblems \eqref{def:tsk-admm} and \eqref{def:tyk-admm}.
Both subproblems is the sum of quadratic function with a convex function (either $g$ or $f$).
For the purpose of this remark, assume that both $f$ and $g$ are simple functions in the sense that
both subproblems can be easily solved if the Hessians of their corresponding quadratic is a multiple of the identity operator.
In most applications, one of the operators is the identity operator which we may assume to be the $D$ operator.
In such a case, subproblem \eqref{def:tyk-admm}  is easy to solve if we choose $H=0$. Also, choosing
$G= \alpha I - \beta C^*C$ where $\alpha$ is such that $G$ is positive semidefinite ensures that the Hessian of the
 quadratic function of \eqref{def:tyk-admm}  is $(1+\mu) \alpha I$, and hence  this subproblem is also easy to solve.

 
The following result, which is the main one of this section, shows that the regularized ADMM class has a better pointwise iteration-complexity than the standard ADMM.
Its complexity bound depends on the quantity
\begin{equation}\label{def:d0admm}
d_0:=\inf_{(s,y,x) \in T^{-1}(0)} \left\{\frac{1}{2}\|s_0-s\|_{\mathcal{S},H}^2+\frac{1}{2}\|y_0-y\|_{\Y,G}^2+\frac{\beta}{2}\|C(y_0-y)\|_\X^2+\frac{1}{2\beta\theta}\|x_0-x\|_\X^2\right\}
\end{equation}
where $T$ is as defined in \eqref{FAB}.

\begin{theorem}\label{th:maintheoADMM} 
DR-ADMM with stepsize $\theta\in (0,(1+\sqrt{5})/2)$  terminates in at most
\begin{equation}\label{eq:iterRADMMtheta<1}
\mathcal{O}\left(\left(1+\frac{\sqrt{d_0}}{\rho}\right)\left[1 + \log^+\left(\dfrac{\sqrt{d_0}}{\rho}\right)\right]\right)
\end{equation} 
iterations with $(s,y,\tilde{x},\tilde b,\tilde a,\tilde q,\tilde p)$ satisfying 
\begin{equation}\label{eq:th_incADMMtheta<1} 
\left( 
\begin{array}{c} 
H\tilde b\\[1mm]  
G\tilde a+C^*\tilde{q}\\[1mm]  
\tilde p
\end{array} 
\right) \in 
\left[ 
\begin{array}{c} 
\partial g(s)- D^*\tilde{x}\\[1mm]  
\partial f(y)- C^*\tilde{x}\\[1mm]  
Cy+Ds-c
\end{array} \right],
\quad  \left(\|\tilde b\|_{\mathcal{S},H}^2+\|\tilde a\|_{\Y,G}^2+\frac{1}{\beta}\norm{\tilde q}_\X^2+\beta\theta\norm{\tilde p}_\mathcal{X}^2\right)^{1/2}\leq \rho.
\end{equation}
\end{theorem}

Before ending this subsection, we compare the  iteration-complexity bound of Theorem~\ref{th:maintheoADMM}
with the  ones obtained in  \cite{He2015,monteiro2010iteration}.
A pointwise convergence rate for an  ADMM scheme is established in \cite{He2015}. More specifically, it is implicitly
shown that  this scheme generates a sequence $\{(s_k,y_k,\tilde x_k,\tilde b_k, \tilde q_k,\tilde p_k)\}$
(which is the same as the one generated
by a cycle of inner iterations of the DR-ADMM with $G=0$, $\theta=1$ and $\mu=0$)
satisfying the inclusion in \eqref{eq:th_incADMMtheta<1}  and
\[
\|\tilde b_k\|_{\mathcal{S},H}^2+ \frac{1}{\beta} \|\tilde q_k\|^2_\X+\beta \| \tilde p_k \|^2_\X  \le \frac{2 d_0}{k} \quad \forall k \ge 1,
\]
which, as a consequence, implies an ${\cal O}(d_0/\rho^2)$ pointwise iteration-complexity bound for finding
a $(s,y,\tilde x,\tilde b, \tilde q,\tilde p) $ satisfying the inequality in \eqref{eq:th_incADMMtheta<1}   with $\theta=1$.
Hence, our bound \eqref{eq:iterRADMMtheta<1} is better than the latter one by an ${\cal O}(\rho \log(\rho^{-1}))$ factor.

We now discuss the 
relationship of bound \eqref{eq:iterRADMMtheta<1} with the ergodic iteration-complexity bound
that immediately follows from \cite[Theorem~7.5]{monteiro2010iteration}).
Given $\rho>0$ and $\varepsilon>0$, this result implies that the standard ADMM (i.e.,   $(H,G)=(0,0)$ and $\mu=0$ in a cycle of inner iterations of the DR-ADMM)  with   $\theta=1$ obtains,
by averaging the first  $j$ elements of the sequence $\{(s_k,y_k,\tilde x_k, \tilde q_k,\tilde p_k)\}$
for some $j={\cal O}(\max\{d_0/\rho,d_0^2/\varepsilon\})$, a quintuple $\{(s^a,y^a,\tilde x^a, \tilde q^a,\tilde p^a)\}$ satisfying
\begin{equation*}
 \left( 
\begin{array}{c} 
0\\[1mm]  
C^*\tilde{q}^a\\[1mm]  
\tilde p^a
\end{array} 
\right) \in 
\left[ 
\begin{array}{c} 
\partial_{\epsilon}g(s^a)- D^*\tilde{x}^a\\[1mm]  
\partial_{\epsilon}f(y^a)- C^*\tilde{x}^a\\[1mm]  
Cy^a+Ds^a-c
\end{array} \right],
\quad  \left(\frac{1}{\beta}\norm{\tilde q^a}_\X^2+\beta\norm{\tilde p^a}_\mathcal{X}^2\right)^{1/2}\leq \rho.
\end{equation*}
Hence, the dependence of the above ergodic bound on $\rho$ is ${\cal O}(d_0/\rho)$, which is better than the
pointwise bound \eqref{eq:iterRADMMtheta<1} by only a logarithmic factor. On the other hand, with respect to $\varepsilon$,
bound \eqref{eq:iterRADMMtheta<1} is better than the above ergodic bound since it does not depend on $\varepsilon$.

Finally, we end this subsection by
mentioning that other  iteration-complexity bounds have been established for the ADMM for any $\theta \in (0,(\sqrt{5}+1)/2)$ based on different
but related stopping criteria (see for example \cite{Gu2015,He2}) and to which the above comments apply in a similar manner.

\subsection{Proof of Theorem~\ref{th:maintheoADMM}}\label{subsec:proofAdmm}

In this subsection, we establish Theorem~\ref{th:maintheoADMM} by first showing that DR-ADMM is an
instance of Framework~3, and then translating Theorem~\ref{th:main} to the context of the DR-ADMM
to obtain the complexity bound \eqref{eq:iterRADMMtheta<1}.

The first result below establishes, as a consequence of some useful relations, that \eqref{eq:th_incADMMtheta<1}  is essentially
an invariance of the inner iterations of the DR-ADMM.

\begin{lemma} \label{pr:aux}
The $k$-th iterate $(s_k,y_k,x_k, \tilde{x}_k)$ obtained during a cycle of inner iterations satisfies
\begin{eqnarray}
0&\in& H(s_k-s_{k-1})+\left[ \partial g(s_k)-D^*\tilde{x}_k+{\mu}H(s_k-s_0)\right], \label{aux.0}\\[3mm]
0&\in&  (G+\beta C^*C)(y_k-y_{k-1})+\left[\partial f(y_k)-C^*\tx_k+\mu(G+\beta  C^*C)(y_k-y_0)\right],\label{aux.2}\\[3mm]
0&=&  \frac{1}{\theta\beta}(x_k-x_{k-1})+\left[Cy_k+Ds_k-c+ \frac{\mu}{\theta\beta} (\tilde x_k-x_0)\right],\label{aux.1}\\[3mm]
\tx_k-x_{k-1}&=&\beta C(y_{k}-y_{k-1}) +\frac{x_k-x_{k-1}}{\theta} \label{aux.3}
\end{eqnarray}
where $\mu$ is  the constant value of the smoothing parameter during this cycle.
As a consequence, the $7$-tuple $(s,y,\tilde x,\tilde b,\tilde a, \tilde q,\tilde p)$ obtained in step 2 of the DR-ADMM
satisfies the inclusion in~\eqref{eq:th_incADMMtheta<1}.
\end{lemma}

\begin{proof}
From the optimality condition of  \eqref{def:tsk-admm}, we have
\begin{equation}\label{eqDStil}
0 \in \partial g(s_k)- D^*(\hat{x}_{k-1}-\beta_1(Cy_{k-1}+Ds_k-c))+(1+\mu)H(s_k-\hat s_{k-1}),
\end{equation}
which, combined with \eqref{defxtk} and  the definition of $\hat s_{k-1}$ in \eqref{defyhk}, yield  \eqref{aux.0}.
Now, from the optimality condition of \eqref{def:tyk-admm} and definition of $u_k$ in \eqref{defuk}, we obtain
\begin{align*}
0 &\in \partial f(y_k)-C^*u_k+\beta_2 C^*(Cy_k+Ds_k-c)+(1+\mu)G(y_k-\hat y_{k-1})\\[2mm]
   &= \partial f(y_k)-C^*[\tilde{x}_{k}+\beta_2 (C\hat{y}_{k-1}+Ds_k-c)]+\beta_2 C^*(Cy_k+Ds_k-c)+(1+\mu)G(y_k-\hat y_{k-1}) \\[2mm]
   & = \partial f(y_k)-C^*\tilde{x}_k+\beta_2 C^*C(y_{k}-\hat{y}_{k-1})+(1+\mu)G(y_k-\hat y_{k-1}).
\end{align*}
Also, the definition of  $\beta_2$ in step 1 of the DR-ADMM and the definition of $\hat{y}_{k-1}$ in \eqref{defyhk} yield
\[
\beta_2(y_k-\hat{y}_{k-1})=\beta(1+\mu)(y_k-\hat{y}_{k-1})=\beta(1+\mu)y_k-\beta(y_{k-1}+\mu y_0) =\beta(y_k-y_{k-1})+\mu\beta(y_k-y_0)
\]
and
\[
(1+\mu)G(y_k-\hat y_{k-1})=(1+\mu)Gy_k-G(y_{k-1}+\mu y_0)=G(y_k-y_{k-1})+\mu G(y_k-y_0).
\]
The last two equalities combined with the previous inclusion prove \eqref{aux.2}.
Moreover, \eqref{aux.1} follows immediately from \eqref{admm:eqxk}.
We now prove relation \eqref{aux.3}. Using the definition of  $\beta_1$ in step 1 of the DR-ADMM
and definitions  $\hat{x}_{k-1}$ and  $\tilde{x}_k$ in \eqref{defyhk} and \eqref{defxtk}, respectively, we obtain
\begin{align*}
\tx_k-x_{k-1}+\frac{\mu}{\theta}(\tx_k-x_0)&=\frac{1}{\theta}\left[(\theta+\mu)\tx_k-\left( \theta x_{k-1}+\mu x_0\right)\right]=\frac{\theta+\mu}{\theta}(\tilde{x}_k-\hat{x}_{k-1})\\[2mm]
&=  -\frac{(\theta+\mu)\beta_1}{\theta}(Cy_{k-1}+Ds_k-c)=  -\beta(Cy_{k-1}+Ds_k-c).
\end{align*}
Identity \eqref{aux.3} now follows by combining the previous relation and \eqref{aux.1}.
Finally, the inclusion in \eqref{eq:th_incADMMtheta<1} follows from  \eqref{aux.0}-\eqref{aux.1} and the  definitions $\tilde b$, $\tilde a$,  $\tilde q$ and  $\tilde p$ in \eqref{eq:rxry}. 
\end{proof}

Define the space
$\mathcal{Z}:=\mathcal{S}\times \Y\times\X$ and endow it with the inner product defined as
\[
\inner{z}{z'} := \langle s,s'\rangle_\mathcal{S}+ \langle y,y'\rangle_\Y+\langle x,x'\rangle_\X
\quad \forall z=(s,y,x), z'=(s,y,x).
\]
Since our approach towards proving Theorem \ref{th:maintheoADMM} is via showing that
DR-ADMM is a special instance of Framework~3, we need to introduce the elements required
by the setting of Section \ref{sec:imp.hpe}, namely, the  distance generating function $w: \Z \to [-\infty,\infty]$,
the convex set $Z \subset \dom w$ and  the seminorm $\|\cdot\|$ on $\Z$. 
Indeed, define $w: \Z \to \R$, the seminorm and set $Z$ as
\begin{equation}\label{df:norm_admm}
 \quad w(z):=\frac{1}{2}\|(s,y,x)\|^2, \quad
 \|z\|:=\left(\|s\|_{\mathcal{S},H}^2+\|y\|_{\Y,G}^2+\beta\|Cy\|_\X^2+\frac{1}{\beta\theta}\|x\|_\X^2\right )^{1/2},
\quad Z:=\Z
\end{equation}
for every $z=(s,y,x) \in \Z$.
 Clearly, the Bregman distance associated with $w$ is given by
\begin{equation}\label{df:dw_admm}
(dw)_z(z')=\frac{1}{2}\|s'-s\|_{\mathcal{S},H}^2+\frac{1}{2}\|y'-y\|_{\Y,G}^2+\frac{\beta}{2}\|C(y'-y)\|_\X^2 +\frac{1}{2\beta\theta}\|x'-x\|_\X^2
\end{equation}
for every $z=(s,y,x) \in \Z$ and $z'=(s',y',x') \in \mathcal{Z}$.
The following result shows that $w$, $\|\cdot\|$ and $Z$ defined above,
as well as the operator $T$ defined in~\eqref{FAB},
fulfill the assumptions of Section~\ref{sec:imp.hpe}.

\begin{lemma}\label{Breges}  Let function $w$, seminorm $\|\cdot\|$ and set $Z$ be as defined above. 
Then, the following statements hold: 
\begin{itemize}
\item [(a)]  the function $w$ is a $(1,1)$-regular distance generating function with respect to $(Z,\|\cdot\|)$;
\item [(b)] the set $Z_{\mu}^*$ as in \eqref{defZu}  where $z_0=(s_0,y_0,x_0)$ and $T$ is  as in  \eqref{FAB} is nonempty for every $\mu >0$.
\end{itemize}
\end{lemma}

\begin{proof}
(a) This statement  follows directly from   Example~\ref{arth2}(b) with $A$ given by
\[
A(s,y,x)=(Hs,(G+\beta C^*C) y,x/(\beta\theta)) \quad \forall (s,y,x)\in \Z.
\]

(b) The proof of this statement  is given in Apendix~\ref{nonemsolu}.
\end{proof}

The next result gives a sufficient condition for the  sequence
generated by a cycle of inner iterations of the DR-ADMM to be an implementation of Framework~2. 

\begin{lemma}\label{lemdeltaxy} Let $\sigma \in [0,1)$ and $\tau \in (0,1)$ be given and
consider the operator $T$  and  Bregman distance $dw$ as in \eqref{FAB} and \eqref{df:dw_admm}, respectively.
Let $\{(s_k,y_k,x_k,\tilde{x}_k,b_k,a_k,q_k,p_k)\}$  be the sequence generated during a cycle of inner iterations of  DR-ADMM with parameter $\D>0$ and define
\begin{equation}\label{def:zk}
 \quad z_{k-1}=(s_{k-1},y_{k-1},x_{k-1}), \quad \tilde{z}_k=(s_k,y_k,\tilde{x}_k), \quad \lambda_k=1, \quad \epsilon_k=0 \quad \forall k\geq 1.
\end{equation}
Then, the following statements hold:
\begin{itemize}
\item [(a)] the sequence $\{(z_k,\tilde{z}_k,\lambda_k,\varepsilon_k)\}$ satisfies  inclusion \eqref{eq:ec.2} and
 the left hand side $r_k$ of this inclusion in terms of $b_k$, $a_k$, $p_k$ and $q_k$ is given by $r_k=(Hb_k,Ga_k+C^*q_k,p_k)$; 
 \item [(b)]  if there exists a  sequence $\{\eta_k\}$ such that  $\{(b_k,a_k,q_k,p_k,\eta_k)\}$ satisfies
 \begin{align*}
[\sigma(1+\theta)-1& ]\frac{\|q_k\|_\X^2}{2\beta\theta}+
\left[\sigma-(\theta-1)^2\right]\frac{\beta\|p_k\|_\X^2}{2\theta}    +\frac{(\sigma +\theta-1)}{\theta}\langle q_k,p_k \rangle_\X \\
& \geq \eta_k-(1-\tau)\eta_{k-1} - \sigma \frac{\|b_k\|_{\mathcal{S},H}^2}{2} - \sigma \frac{\|a_k\|_{\Y,G}^2}{2},
 \end{align*}
 then the sequence $\{(z_k,\tilde{z}_k,\lambda_k,\varepsilon_k,\eta_k)\}$ satisfies the error condition \eqref{eq:es.2};
\item [(c)] condition \eqref{innerstop} is equivalent to $\|r_k\|^*\leq\rho/2$  where 
$r_k$ is as in \eqref{eq:ec.2}  and $\|\cdot\|$ is the seminorm defined in \eqref{df:norm_admm}.
\end{itemize}
As a consequence, if  the assumption of  (b) is satisfied then the sequence $\{(z_k,\tilde{z}_k,\lambda_k,\varepsilon_k,\eta_k)\}$ is an implementation of   Framework 2  with input $z_0=(s_0,y_0,x_0)$, $(\eta_0, \D)$ and $\lambda=1$.

Moreover, if every cycle of inner iterations of DR-ADMM  satisfies the assumption of  (b), then DR-ADMM is an instance of Framework~3.
\end{lemma}
\begin{proof}
(a) These statements  follows from relations    \eqref{aux.0}-\eqref{aux.1} and 
definitions  in \eqref{defr}, \eqref{FAB} and \eqref{df:dw_admm}.

(b) Using  \eqref{defr} and  \eqref{aux.3}, we obtain
\begin{equation}\label{eq:deltas}
x_{k-1}-\tilde{x}_k= q_k+\beta p_k,\quad x_k-\tilde x_k =q_k +(1-\theta)\beta p_k.
\end{equation}
Hence, it follows from \eqref{df:dw_admm} and \eqref{def:zk} that 
\begin{align*}\label{eq:lefthpecond_eta}
(dw)_{ z_{k}}(\tilde{z}_k)  +\lambda_k\varepsilon_k&=\frac{1}{2\beta\theta}\|\tx_k-x_k\|_\X^2 \\
 &=\frac{1}{2\beta\theta}\left[\|q_k\|_\X^2+\beta^2(\theta-1)^2\|p_k\|_\X^2+ 2(1-\theta)\beta\langle q_k,p_k \rangle_\X\right].
\end{align*}                                             
On the other hand, using  \eqref{df:dw_admm}-\eqref{eq:deltas} and definitions of $b_k$, $a_k$, $q_k$ in \eqref{defr}, we obtain           
\begin{equation*}\label{eq:righpecond_eta}
\begin{array}{rcl} 
 (dw)_{ z_{k-1}}(\tilde{z}_k)     &=& \frac{1}{2}\|s_{k-1}-s_k\|_{\mathcal{S},H}^2+\frac{1}{2}\|y_{k-1}-y_k\|_{\Y,G}^2+ \frac{\beta}{2} \|C(y_{k-1}-y_{k})\|_\X^2+\frac{1}{2\beta\theta}\|x_{k-1}-\tilde x_k\|_\X^2 \\[3mm]
                                                                                          &=&\frac{1}{2}\|b_k\|_{\mathcal{S},H}^2+\frac{1}{2}\|a_k\|_{\Y,G}^2+\frac{1}{2\beta}\|q_k\|_\X^2+\frac{1}{2\beta\theta}\| q_k+\beta p_k\|_\X^2\\[3mm]
                                                                                          &=&\frac{1}{2}\|b_k\|_{\mathcal{S},H}^2 +\frac{1}{2}\|a_k\|_{\Y,G}^2+\frac{(\theta+1)}{2\beta\theta}\|q_k\|_\X^2+\frac{ \beta}{2\theta}\|p_k\|_\X^2+\frac{1}{\theta}\langle q_k,p_k\rangle_\X.
\end{array}
\end{equation*} 
Statement (b) now follows immediately from the above two identities.

(c) First of all, note that $\|\cdot\|^2= \inner{A(\cdot)}{\cdot}$ where $A(s,y,x)=(Hs,(\beta C^*C+G) y,x/(\beta\theta))$ for every $(s,y,x)\in~\Z$. Hence, it follows from the identity in (a),   \eqref{defr}, the definition of $A$ and  Proposition~\ref{propdualnorm}(a) that
\begin{align*}
\|r_k\|^*&=\|(Hb_k,Ga_k+C^*q_k,p_k)\|^*=\|A(b_k,a_k,x_{k-1}-x_k)\|^*  \\
&=  \|(b_k,a_k,x_{k-1}-x_k)\| =\left(\|b_k\|_{\mathcal{S},H}^2+\| a_k\|_{\Y,G}^2+ \frac{1}{\beta} \norm{q_k}_\X^2+\beta\theta\|p_k\|_\X^2\right)^{1/2}
\end{align*}
from which statement (c) follows. 

To show the last statement of the lemma,
first note that item (a),  \eqref{defr}, \eqref{eq:rxry} and \eqref{df:dw_admm}
 imply that $\tilde r = (H\tilde b, G\tilde a+C^*\tilde q, \tilde p)$. Now, a similar argument as above using definition of $\tilde p$ in 
 \eqref{eq:rxry},  the definition of $A$,  Proposition~\ref{propdualnorm}(a) imply
 that $\|\tilde r\|^* = (\|\tilde b\|_{\mathcal{S},H}^2+\| \tilde a\|_{\Y,G}^2+\frac{1}{\beta}\norm{\tilde q}_\X^2+\beta\theta\norm{\tilde p}_\X^2)^{1/2}$,
from which the last statement of the lemma follows.
\end{proof}

In view of Lemma \ref{lemdeltaxy}, it suffices to show that DR-ADMM satisfies the assumption of Lemma~\ref{lemdeltaxy}(b)
in order to show that it is an instance of Framework~3.
We will prove the latter fact by considering two cases, namely, whether the stepsize
$\theta$ is in $(0,1)$ or in $[1,(\sqrt{5}+1)/2)$. The next result consider the case in
which $\theta \in (0,1)$ and Lemma \ref{th:hpeheta>1} below considers the other case.

\begin{lemma}\label{th:hpeheta<1} Assume that the DR-ADMM stepsize $\theta\in (0,1)$.  Let  $\{(b_k,a_k,q_k,p_k)\}$ be the sequence generated during a cycle of inner iterations of  DR-ADMM with parameter $\D>0$ and define $\eta_k=0$ for every $k\geq 0$.
Then,  the sequence $\{(b_k,a_k,q_k,p_k,\eta_k)\}$
satisfies the assumption of Lemma~\ref{lemdeltaxy}(b) with  $\sigma=\theta+(\theta-1)^2\in (0,1)$ and any $\tau\in (0,1)$.
\end{lemma}

\begin{proof} Using the definition of $\sigma$, we  have
\begin{align*}\label{eq:desMatrizM}
 [\sigma(1&+\theta) -1] \frac{\|q_k\|_\X^2}{2\beta\theta}  +  \left[\sigma-(\theta-1)^2\right]\frac{\beta\|p_k\|_\X^2}{2\theta}     +\frac{(\sigma +\theta-1)}{\theta}\langle q_k,p_k \rangle_\X\\&= 
\frac{\theta^2}{2\beta} \|q_k\|_\X^2+\frac{\beta\|p_k\|_\X^2}{2}   +\theta\langle q_k,p_k \rangle_\X= 
\frac{1}{2\beta}\|\theta q_k+\beta p_k\|_\X^2\geq 0.
 \end{align*} 
 Hence the lemma follows  due to the definition of $\{\eta_k\}$.
 \end{proof}
 Before handling the other case in which  $\theta \in [1,(\sqrt{5}+1)/2)$, we  first establish  the following technical result. 
\begin{lemma}\label{lem:deltak}Consider the sequence $\{(y_k,x_k,\tilde{x}_k,a_k,q_k,p_k)\}$ generated during a cycle of inner iterations of  DR-ADMM with parameter $\D>0$. Then, the following statements hold:
\begin{itemize}
\item [(a)] if  $\theta \in [1,2)$, then 
$$
\|q_1\|_\X\|p_1\|_\X +\|a_1\|_{\Y,G}^2 \leq \frac{8\theta\max\{\beta,\beta^{-1}\}}{2-\theta}d_0
$$ 
where $d_0$ is as in \eqref{def:d0admm};
\item [(b)] for any  $k\geq 2$, we have
\begin{equation}\label{eq:deltak}
\langle q_k,p_k\rangle_\X\geq (1-\theta)\langle q_k, p_{k-1}\rangle_\X+\frac{1}{2}\|a_k\|_{\Y,G}^2-\frac{1}{2}\|a_{k-1}\|_{\Y,G}^2.
\end{equation}

\end{itemize}
\end{lemma}
\begin{proof}
The proof of this lemma is given in Appendix~\ref{nodo}.
\end{proof}

In contrast to the case in which $\theta \in (0,1)$, the following result shows that the case in
which  $\theta \in [1,(\sqrt{5}+1)/2)$ requires a non-trivial choice of sequence $\{\eta_k\}$,
and hence uses the full generality of the approach of Section \ref{sec:imp.hpe}.

\begin{lemma} \label{th:hpeheta>1}
Assume that the DR-ADMM stepsize $\theta\in [1,(\sqrt{5}+1)/2)$ and consider the sequence  $\{(b_k,a_k,q_k,p_k)\}$  generated during a cycle of inner iterations of DR-ADMM with parameter $\D>0$.
Then,  there exist $\sigma \in [1/2,1)$ and $\tau \in (0,1/2]$ such that the sequence $\{(b_k,a_k,q_k,p_k,\eta_k)\}$ where $\{\eta_k\}$ is defined as 
\begin{equation*}
 \eta_0=\frac{4(\sigma+\theta-1)\max\{\beta,\beta^{-1}\}}{(2-\theta)(1-\tau)}d_0,\quad    \eta_k= \frac{[\sigma-(\theta-1)^2]\beta}{2\theta}\|p_k\|_\X^2+\frac{\sigma+\theta-1}{2\theta(1-\tau)}\|a_k\|_{\Y,G}^2 \quad \forall k\geq 1
\end{equation*}
satisfies the assumption of  Lemma~\ref{lemdeltaxy}(b). 
\end{lemma}
\begin{proof} It follows from definition of  $\eta_1$ and the Cauchy-Schwarz inequality that 
\begin{align*}
&[\sigma(1+\theta)-1 ]\frac{\|q_1\|_\X^2}{2\beta\theta}+\left[\sigma-(\theta-1)^2\right]\frac{\beta\|p_1\|_\X^2}{2\theta} +\sigma \frac{\|a_1\|_{\Y,G}^2}{2}  +\frac{(\sigma +\theta-1)}{\theta}\langle q_1,p_1 \rangle_\X \\
&\geq [\sigma(1+\theta)-1 ]\frac{\|q_1\|_\X^2}{2\beta\theta}+
\left[\sigma+\frac{(\sigma+\theta-1)(1-2\tau)}{\theta(1-\tau)}\right] \frac{\|a_1\|_{\Y,G}^2}{2}+ \eta_1 \\
& -\frac{(\sigma +\theta-1)}{\theta}\left( \|q_1\|_\X\|p_1\|_\X+\|a_1\|_{\Y,G}^2\right)\\
&\geq \eta_1  -\frac{(\sigma +\theta-1)}{\theta}\left( \|q_1\|_\X\|p_1\|_\X+\|a_1\|_{\Y,G}^2\right) \\
&  \geq \eta_1-(1-\tau)\eta_{0}
\end{align*}
where the second inequality holds for any   $\theta \in [1,2)$, $\sigma \ge 1/2$ and  $ \tau \in (0,1/2]$ and the 
third inequality is due to  Lemma~\ref{lem:deltak}(a) and definition of $\eta_0$.
Therefore, the inequality in Lemma~\ref{lemdeltaxy}(b) with $k=1$ holds for any $\theta$, $\sigma$ and $\tau$ as above.

Also, using Lemma~\ref{lem:deltak}(b) and the definition of $\{\eta_k\}$ in the statement of the lemma,
and performing some algebraic manipulations, we easily see 
that a sufficient condition for the inequality in Lemma~\ref{lemdeltaxy}(b)  with $ k\geq 2$ to hold is that
\begin{align*}
 &(\sigma(1+\theta)-1)\frac{\|q_k\|_\X^2}{2\beta}+(1-\tau)[\sigma-(\theta-1)^2]\frac{\beta\|p_{k-1}\|_\X^2}{2}  +
[\theta\sigma-\tau(\sigma+ \theta + \sigma \theta -1)]\frac{\|a_k\|_{\Y,G}^2}{2(1-\tau)}\\
 &+(\sigma + \theta-1)(1-\theta)\langle q_k ,p_{k-1}\rangle_\X\geq 0.
\end{align*}
Clearly, in view of the Cauchy-Schwarz inequality, the above inequality holds if the matrix
\begin{equation*} \label{matrixtheta>1}
M(\theta,\sigma,\tau)= \left[
\begin{array}{ccc} 
 \sigma(1+\theta)-1& (\sigma+\theta-1)(1-\theta)&0 \\[2mm]
(\sigma+\theta-1)(1-\theta)&  (1-\tau)[\sigma-(\theta-1)^2]   &0\\[2mm]
0&0&{\theta\sigma-\tau(\sigma+ \theta + \sigma \theta -1)}
\end{array} \right]
\end{equation*}
is positive semidefinite.
Since for any $\theta\in [1,(\sqrt{5}+1)/2)$, the matrix $M(\theta,1,0)$ is easily seen to be  positive definite, it follows that
there exist $\sigma \in [1/2,1)$ close to $1$ and $\tau \in (0,1/2]$ close to $0$ such that the above matrix is positive semidefinite.
We have thus established that the conclusion of the lemma holds.
\end{proof}

  Now we are ready to  prove Theorem~\ref{th:maintheoADMM}.
  \\[2mm]
  {\bf Proof of Theorem~\ref{th:maintheoADMM}:}  First note that \eqref{eq:th_incADMMtheta<1} follows from the last statement in Lemma~\ref{pr:aux} and \eqref{Scri2}. Now,
   it follows by combining Lemmas \ref{lemdeltaxy}, \ref{th:hpeheta<1} and \ref{th:hpeheta>1}  that DR-ADMM with $\theta\in(0,(1+\sqrt{5})/2)$ is an instance of Framework~3
  applied to problem \eqref{FAB} in which $\{(z_k,\tilde{z}_k,\lambda_k,\varepsilon_k)\}$ is as in \eqref{def:zk} and  $\{\eta_k\}$ is as defined in  Lemma~\ref{th:hpeheta<1} if $\theta\in (0,1)$ or as in Lemma~\ref{th:hpeheta>1}
  if $\theta\in[1,(1+\sqrt{5})/2)$. This conclusion together with Lemma \ref{Breges}(a) then imply that Theorem~\ref{th:main},  as well as the observation following it, holds with $\lambda=m=M=1$. Hence,
 estimate in  \eqref{eq:iterRADMMtheta<1} now follows from the latter observation and definition of $\eta_0$.
 
\appendix
\noindent
\\
{\bf \Large{ Appendix}}
\noindent
\section{Some basic technical results}\label{basiresul}

The following result gives some properties of the dual seminorm whose simple proof is  omitted.
 
\begin{proposition}\label{propdualnorm} 
Let $A:\Z\to \Z$ be a self-adjoint positive semidefinite linear operator and consider
 the seminorm $\|\cdot\|$ in $\Z$ given by $\|z\|= \langle A z, z\rangle ^{1/2}$ for every $z\in \Z$. Then, the following statements  hold:
\begin{itemize}
\item[(a)] $\dom  \|\cdot\|^*=\mbox{Im}\;(A)$ and $\| Az\|^*=\|z\|$ for every $z\in \Z$;
\item[(b)] if $A$ is invertible then $\|z\|^*= \langle A^{-1} z, z\rangle ^{1/2}$ for every $z\in \Z$.
\end{itemize}
\end{proposition}


The following result discusses the existence of solution for a certain monotone inclusion problem.

\begin{proposition}\label{existence} 
Let   $T:\Z\tos \Z$ be a maximal monotone operator and   $w : \Z  \to [-\infty,\infty]$  be a
distance generating function such that $ \inte (\dom w) \supset \ri (\dom T)$ and
$w$ is  strongly convex on $ \ri (\Dom (T)$.
Then, for every $\hat z \in  \inte(\dom w)$, the inclusion
$ 0\in (T+ \partial (d{w})_{\hat z})(z)$ has a unique solution $z$, which must necessarily be on
$\Dom (T) \cap \inte(\dom w)$ and hence satisfy the inclusion
$0  \in (T+ \nabla (d{w})_{\hat z})(z)$.
\end{proposition}
\begin{proof}
Let $\hat z   \in  \inte(\dom w)$  be given. First note that 
\[
\inte( \dom w) \subset \Dom (\partial w) = \Dom ( \partial (dw)_{\hat z} ) \subset \dom w,
\]
from which we conclude that $\ri (\Dom ( \partial (dw)_{\hat z} ) = \inte( \dom w)$. Moreover, by
Proposition~2.40 and Theorem~12.41 of \cite{VariaAna}, we have that
$\mbox{ri}(\Dom(T))\neq \emptyset$. These two observations then imply that
\begin{equation} \label{eq:veri-inte}
\ri ( \Dom (T)) \cap \ri( \Dom ( \partial (dw)_{\hat z} ) = \ri(\Dom (T)) \cap \inte (\dom w) = \ri ( \Dom (T) ) \ne \emptyset.
\end{equation}
Clearly, $(dw)_{\hat z}(\cdot)$ is a proper lower semicontinuous function due to Definition \ref{defw0} and \eqref{def_d},
and hence $\partial (dw)_{\hat z}$ is maximal monotone in view of Theorem~12.17 of \cite{VariaAna}.
Thus, it follows from \eqref{eq:veri-inte}, the last conclusion, the assumption that $T$ is maximal monotone and
Corollary~12.44 of \cite{VariaAna} that $T + \partial (d{w})_{\hat z}$  is maximal monotone.
Moreover, since $w$ is strongly convex on $\ri(\Dom(T))$,  it follows that $\partial (d{w})_{\hat z}$
is strongly monotone on $\ri(\Dom(T))$.  By using a simple limit argument
and the fact that $\partial w$ is a continuous map on $\inte(\dom w)$ due to Definition \ref{defw0}, we conclude that
$\partial (d{w})_{\hat z}$
is strongly monotone on the larger set $\Dom(T) \cap \inte(\dom w)$.
Since the latter set is equal to $\Dom ( T+ \partial (d{w})_{\hat z})$ and $T$ is monotone,
we conclude that $ T+ \partial (d{w})_{\hat z}$ is strongly monotone.
The first conclusion of the proposition now  follows  from Proposition~12.54 of \cite{VariaAna}.
The second conclusion follows immediately from the first one and the fact that, by  Definition \ref{defw0},
$\Dom (\partial w) = \inte( \dom w)$.
\end{proof}

Next proposition discuss the existence of solutions of a problem related to \eqref{FAB}.

\begin{proposition}\label{lem:auxSol}
Let a linear operator $E:\Z\to\tilde \Z$, a vector $e \in {\rm Im}\, E$  and a  proper closed convex function $h:\mathcal{Z}\to \R\cup\{+\infty\}$ be such that 
\begin{equation} \label{prob:aux_apendix}
 \inf_{z\in \Z} \left \{ h(z): Ez=e \right\}<\infty.
\end{equation}
 Then, the following statements hold:
\begin{itemize}
\item [(a)] if $\ri (\dom h^*) \cap {\rm Im } \,(E^*)\neq \emptyset$,  then \eqref{prob:aux_apendix} has an optimal solution $z^*$;
\item [(b)] the optimal solution set of \eqref{prob:aux_apendix} is nonempty and bounded if and only if $0\in \inte  (\dom h^* + {\rm Im}\,(E^*))$;
\item[(c)] if the assumption of (a) holds and \eqref{prob:aux_apendix} has a Slater point, i.e., a point $\bar z \in \ri(\dom h)$ such that
$E\bar z=e$, then there exists a Lagrange pair $(z,x)=(z^*,x^*)$ satisfying
\[
0 \in \partial h(z)-E^* x, \quad Ez=e.
\]
\end{itemize} 
As a consequence of (b), if the set of optimal solutions of \eqref{prob:aux_apendix} is nonempty and bounded, then $\ri (\dom h^*) \cap {\rm Im}\, (E^*)\neq \emptyset$.
\end{proposition}
\begin{proof}
Since $h$ is a proper closed convex function, we have $(h^*)^*=h$. Hence,  
the proof of  (a) and (b)  follows from Lemma~2.2.2 in Chap.~X of  \cite{Hiriart2} 
with $A_0=E^*$, $g=h^*$ and $s=e$, and the discussion following Theorem~2.2.3. 
The proof of (c) follows easily from  \cite[Corollary~28.2.2]{Rockafellar70}.
\end{proof}


\begin{corollary} \label{cr:auxSol}
Let $E$, $e$ and $h$ be as in Proposition \ref{lem:auxSol} and assume that $\ri (\dom h^*) \cap {\rm Im } \,(E^*)\neq \emptyset$.
Then, the problem
\begin{equation} \label{prob:aux_apendix1}
 \inf_{z\in \Z} \left \{ h(z) + \frac12 \|Ez-e\|^2  \right\}
\end{equation}
has an optimal solution.
\end{corollary}
\begin{proof}
The proof follows by noting that the above problem is equivalent to
\[
 \inf_{z\in \Z} \left \{ h(z) + \frac12 \|w\|^2 : Ez-w = e  \right\}
\]
and by applying Proposition \ref{lem:auxSol}(a) to the latter problem.
\end{proof}

\section{Proof of Proposition \ref{prop:ergodic}}\label{proofpropergodic}
First, let us show that for every $k> \ell \geq 0$, the following inequalities  hold:
 \begin{equation}\label{ineq:epspropergodic} 
 \Lambda_{\ell,k} \, \varepsilon^a_{\ell,k} \leq (dw)_{z_{\ell}}(\tilde z^a_{l,k})+\eta_{\ell}\leq \displaystyle\max_{i=\ell+1,\ldots, k}(dw)_{z_{\ell}}(\tilde z_{i})+\eta_{\ell}.
 \end{equation}
Indeed, it follows from Lemma~\ref{lema_desigualdades}(b) that  for every $z \in \dom w$,
\[(dw)_{z_{\ell}}(z) - (dw)_{z_k}(z)+ (1-\tau) \eta_{\ell} \geq (1 - \sigma)\sum_{i=\ell+1}^k (dw)_{z_{i-1}}(\tilde{z}_i) + \sum_{i=\ell+1}^k\lambda_i (\langle r_i, \tilde{z}_i - z\rangle + \varepsilon_i).\]
Taking $z=\tilde z^a_{l,k}$ in the last inequality and using that $\tau>0$ and $\sigma <1$, we have
\[(dw)_{z_{\ell}}(\tilde z^a_{l,k}) +\eta_{\ell}
\geq \sum_{i=\ell+1}^k\lambda_i (\langle r_i, \tilde{z}_i - \tilde z^a_{l,k}\rangle + \varepsilon_i),\]
which combined with definition of $\varepsilon^a_{\ell,k}$ in \eqref{SeqErg}, proves the first inequality of the lemma.  The second inequality of the lemma follows from definition of $\tilde z^a_{l,k}$ and  convexity of the function $(dw)_{z_\ell}(\cdot)$.

The proof of  \eqref{qer10}  follows  from  the {\it weak transportation formula}
\cite[Theorem~2.3]{Bu-Iu-Sv:teps} and the fact that   $r_k \in (S+T^{[\varepsilon_k]})(\tz_k) \subset (S+T)^{[\varepsilon_k]}(\tz_k)$. Now,  \eqref{grad-d}, \eqref{subpro} and \eqref{SeqErg} imply that
\begin{equation}\label{eq:estimate_rakl}
 \Lambda_{\ell,k}r^a_{\ell,k}=\sum_{i=\ell+1}^k \lambda_i r_i= \sum_{i=\ell+1}^k \nabla(dw)_{z_{i}}({z}_{i-1}) =  \nabla (dw)_{z_{k}}({z}_{\ell}).
\end{equation}
Also, \eqref{grad-d}, \eqref{eq:789} and \eqref{eq:094} yield
\begin{align*}
\|\nabla (dw)_{z_{k}}({z}_\ell)\|^* &\leq  \| \nabla (dw)_{z_{k}}(z^{*})\|^*+\| \nabla (dw)_{z_\ell}(z^*)\|^*
  \leq   \frac{\sqrt{2}M}{\sqrt{m}}
\left[  ((dw)_{z_{k}}(z^{*}))^{1/2}+((dw)_{z_{\ell}}(z^*))^{1/2}\right] \\
&\leq \frac{\sqrt{2}M}{\sqrt{m}}\left[(1-\underline{\tau})^{(k-\ell)/2}+1\right]\left[(dw)_{z_{\ell}}(z^*)+\eta_{\ell}\right]^{1/2}\leq \frac{2\sqrt{2}M}{ \sqrt{m}}\left[(dw)_{z_{\ell}}(z^*)+\eta_{\ell}\right]^{1/2}.
\end{align*}
Combining the last inequality with \eqref{eq:estimate_rakl}  and using relation \eqref{eq:094} with $k=\ell$ and $\ell=0$, we obtain
\[
 \|r^a_{\ell,k}\|^* \le \frac{2\sqrt{2}M}{ \Lambda_{\ell,k}\sqrt{m}}(1 - \underline{\tau})^{\ell/2}\left[((dw)_{z_0}(z^*))+\eta_0\right]^{1/2}.
\]
Hence, since $z^*$ is an arbitrary point in $(S+T)^{-1}(0)$, the first bound in \eqref{ad901} is proved. Let us now prove  the second bound in \eqref{ad901}. It follows from   property \eqref{lipsc}, triangle inequality and  $2ab\leq a^2+b^2$ for all $ a,b\geq 0$ that, for every $i>\ell\geq 0$, 
\[
(dw)_{z_{\ell}}(\tilde z_{i})\leq \frac{M}{2}\|\tilde z_{i}- z_{\ell}\|^2\leq  \frac{M}{2}(\|\tilde z_{i}- z_{i}\|+\|z_{i}- z^*\|+\| z^*- z_{\ell}\|)^2\leq   \frac{3M}{2} (\|\tilde z_{i}- z_{i}\|^2+\|z_{i}- z^*\|^2+\| z^*- z_{\ell}\|^2).
\]
 Hence,  property  \eqref{strongly}  and the fact that $m \leq M$ (see \eqref{lipsc}) yield
\[
(dw)_{z_{\ell}}(\tilde z_{i})+\eta_{\ell}\leq  \frac{3M}{m}\left((dw)_{z_{i}}(\tilde z_{i})+(dw)_{z_{i}}(z^*)+(dw)_{z_{\ell}}( z^*)+\eta_{\ell}\right), \quad \forall \;i>\ell\geq 0.
\]
Since \eqref{eq:094} implies that the sequence $\{(dw)_{z_{i}}(z^*)+\eta_{i}\}$ is non-increasing, it follows from the last inequality and  Lemma~\ref{lema_desigualdades}(d) that
\begin{align*}
(dw)_{z_{\ell}}(\tilde z_{i})+\eta_{\ell}\leq\; &  \frac{3M}{m}\left(\frac{1}{ 1 - \sigma} +2\right)[(dw)_{z_{\ell}}(z^*)+\eta_\ell]
\leq   \frac{9M}{m (1 - \sigma)}[(dw)_{z_{\ell}}(z^*)+\eta_\ell], \quad \forall \;i>\ell\geq 0,
\end{align*}
which, combined with \eqref{ineq:epspropergodic} and \eqref{eq:094}, yields
\[
\varepsilon^a_{\ell,k} \leq \frac{9M}{ m(1 - \sigma)\Lambda_{\ell,k} }
(1 - \underline{\tau})^{\ell}\left[(dw)_{z_{0}}(z^*)+\eta_0\right]
\]
Since $z^*$ is an arbitrary point in $(S+T)^{-1}(0)$, the proof is concluded.  \hfill{ $\square$}

\section{Proof of Lemma~\ref{Breges}(b)}\label{nonemsolu}
Let a scalar $\mu>0$ and note  that  
 $(s,y,x)\in Z_{\mu}^*$ if and only if  $(s,y,x)$  satisfies 
\begin{equation}\label{aux.121}
0 \in \partial g(s)-D^*x+{\mu}H(s-s_0),\;
0 \in  \partial f(y)-C^*x+\mu\beta  C^*C(y-y_0), \;
Cy+Ds-c+\frac{\mu}{\beta\theta}(x-x_0)=0.
\end{equation}
Hence, the proof of Lemma~\ref{Breges}(b) will follow if we show that  \eqref{aux.121} has a solution. Towards this goal, let us
consider the problem
\begin{equation}\label{prob:perturbedapendix} 
\inf_{(s,y,u)} \left\{g(s)+f(y)+\frac{\mu}{2}\|s-s_0\|^2_{\mathcal{S},H}
+\frac{\mu\beta}{2}\|C(y-y_0)\|^2_\X+\frac{\beta\theta}{2\mu}\|u+\frac{\mu}{\beta\theta}x_0\|_\X^2: Ds+Cy+u=c\right\}.  
\end{equation}
It is easy to see that $(s,y,c-Ds-Cy)$ is a Slater point of the above problem for any   
$(s,y)\in \ri (\dom g)\times \ri (\dom f)\neq \emptyset$.
Hence, since condition~{\bf B2} easily implies that the assumption of Proposition~\ref{lem:auxSol}(a) holds,
 it follows from  Proposition~\ref{lem:auxSol}(c) that problem \eqref{prob:perturbedapendix} 
 has a solution $(\bar{s},\bar{y},\bar{u})\in\mathcal{S}\times\Y\times\X$  and  an associated Lagrange multiplier $\bar{x}\in \X$.
The latter conclusion and the optimality conditions for \eqref{prob:perturbedapendix}
 immediately imply that $(\bar{s},\bar{y},\bar{x})$ satisfies \eqref{aux.121}.  \hfill{ $\square$}


\section{Proof of Lemma~\ref{lem:deltak}}\label{nodo}
(a) Let  a point  $z_{\mu}^*:=(s_{\mu}^*,y_{\mu}^*,x_{\mu})\in Z_{\mu}^*$ (see Lemma~\ref{Breges}(b)) and consider $(z_0, z_1,\tilde{z}_1,\lambda_1,\varepsilon_1)$ as in \eqref{def:zk}. It follows from the definitions of $p_1$ and $q_1$,    $2ab\leq a^2+b^2$ for all $ a,b\geq 0$ and $\theta \geq 1$ that
\begin{equation}\label{eq_00000001}
\begin{array}{rcl}
 \|p_1\|_\X\|q_1\|_\X &= &\frac{1}{\theta}\|x_1-x_0\|_\X\|C(y_1-y_0)\|_\X
 \leq \frac{1}{2\theta}\left(\|x_1-x_0\|_\X^2+\|C(y_1-y_0)\|_\X^2\right)\\[2mm]
 &\leq &\frac{1}{\theta}{\|x_1-x_{\mu}^*\|_\X^2}+\|C(y_1-y_{\mu}^*)\|_\X^2+\frac{1}{\theta}{\|x_0-x_{\mu}^*\|_\X^2}+\|C(y_0-y_{\mu}^*)\|_\X^2\\
\end{array}
\end{equation}
Using the definition of $a_1$ and simple calculus, we obtain
\[
\|a_1\|_{\Y,G}^2= \|y_1-y_0\|_{\Y,G}^2\leq 2(\|y_1-y_{\mu}^*\|_{\Y,G}^2+\|y_0-y_{\mu}^*\|_{\Y,G}^2)
\]
which, combined with \eqref{eq_00000001} and definitions of $z_0,z_1$ and $dw$ (see \eqref{df:dw_admm}), 
yields
\begin{equation}\label{eq_000000012}
 \|p_1\|_\X\|q_1\|_\X+\|a_1\|_{\Y,G}^2\leq  4\max\{\beta,\beta^{-1}\}\left((dw)_{z_1}(z_{\mu}^*)+(dw)_{z_0}(z_{\mu}^*)\right).
\end{equation}
On the other hand, 
 Lemma~\ref{lemdeltaxy}(a) implies that  inclusion \eqref{eq:ec.2} is satisfied for $(z_0, z_1,\tilde{z}_1,\lambda_1,\varepsilon_1)$ and hence it follows from Lemma~\ref{lema_desigualdades}(a) with $z=z^*_\mu$, $\lambda_1=1$ and  the fact that $\langle r_1, \tilde{z}_1 - z_\mu^*\rangle\geq 0$ (see \eqref{a390} with $k=1$, $z^*=z_{\mu}^*$ and $\varepsilon_1=0$)
 that
\begin{equation}\label{ineq_lema007d1d0}
(dw)_{z_1}(z_{\mu}^*)\leq (dw)_{z_0}(z_{\mu}^*) +(dw)_{z_1}(\tilde{z}_1)-(dw)_{z_0}(\tilde{z}_1).
\end{equation}
Using the definitions in  \eqref{df:dw_admm}, \eqref{def:zk}, \eqref{defr}  and  \eqref{eq:deltas}, we obtain
\begin{align*}
(dw)_{z_1}(\tilde{z}_1)-(dw)_{z_0}(\tilde{z}_1)&=
\frac{1}{2\beta\theta}\|q_1+(1-\theta)\beta p_1\|_\X^2-  \frac{1}{2}{\|s_1\|_{\mathcal{S},H}^2}-\frac{1}{2}{\|a_1\|_{\Y,G}^2}\\
&-\frac{1}{2\beta}\|q_1\|_\X^2-
\frac{1}{2\beta\theta}\|q_1+\beta p_1\|_\X^2\\
&= \frac{(\theta-1)\beta}{2}\|p_1\|_\X^2-\frac{1}{2}\left\|\sqrt{\beta} p_1 +\frac{q_1}{\sqrt{\beta}}\right\|_\X^2-  \frac{1}{2}{\|s_1\|_{\mathcal{S},H}^2}-\frac{1}{2}{\|a_1\|_{\Y,G}^2} \\
&\leq \frac{(\theta-1)\beta}{2}\|p_1\|_\X^2\leq \frac{(\theta-1)}{\theta}\left( \frac{\|x_1-x_{\mu}^*\|_\X^2}{\beta\theta}+\frac{\|x_0-x_{\mu}^*\|_\X^2}{\beta\theta}\right)\\
&\leq \frac{2(\theta-1)}{\theta} \left[(dw)_{z_1}(z_{\mu}^*)+(dw)_{z_0}(z_{\mu}^*)\right]
\end{align*}
where   the second inequality is due to the definition of $p_1$  and the fact that $2ab\leq a^2+b^2$ for all $ a,b\geq 0$, and the last inequality is due to \eqref{df:dw_admm} and definitions of $z_0,z_1$ and $z_{\mu}^*$.
Hence, combining the last estimative with \eqref{ineq_lema007d1d0},   we obtain 
$$
(dw)_{z_1}(z_{\mu}^*)\leq \frac{\theta}{2-\theta}\left(1+\frac{2(\theta-1)}{\theta}\right)(dw)_{z_0}(z_{\mu}^*)=\frac{3\theta-2}{2-\theta}(dw)_{z_0}(z_{\mu}^*).
$$ 
Therefore, statement (a) follows from \eqref{eq_000000012}, the last inequality and Lemma~\ref{lm:dz}.
 
(b) From the inclusion  \eqref{aux.2} we see that
$$
 C^*(\tilde{x}_j-\beta C(y_j-y_{j-1}))-G(y_j-y_{j-1})\in \partial f_{\mu,\beta}(y_j) \qquad  \forall  j\geq1
$$
where $f_{\mu,\beta}(y):=f(y) + (\beta\mu/2)\|C(y-y_0)\|_\X^2+ (\mu/2)\|y-y_0\|_{\Y,G}^2$  for every $y\in \Y$. Hence, using  relation \eqref{aux.3}, we have
\begin{equation}\label{eq:auxlemadmm}
\frac{1}{\theta} C^*(x_j-(1-\theta)x_{j-1})-G(y_j-y_{j-1})\in \partial f_{\mu,\beta}(y_j) \qquad \forall j\geq1.
\end{equation}
For every $k\geq 2$, using the previous inclusion for $j=k-1$ and $j=k$, it follows from the monotonicity of the subdifferential of $f_{\mu,\beta}$ and the definitions of  
$p_k$, $a_k$ and $q_k$ that
\begin{align*}
0&\leq \left \langle \frac{1}{\theta} C^*(x_k-x_{k-1})- \frac{(1-\theta)}{\theta} C^*(x_{k-1}-x_{k-2})- G(y_k-y_{k-1})+ G(y_{k-1}-y_{k-2}),y_k-y_{k-1}\right\rangle_\Y \\
  &=  \langle \beta C^*(p_k-(1-\theta)p_{k-1}),y_{k-1}-y_{k}\rangle_\Y- \langle Ga_k- Ga_{k-1},a_k  \rangle_\Y\\
  &=\langle p_k-(1-\theta)p_{k-1},q_k\rangle_\Y-\|a_k\|_{\Y,G}^2+ \langle G a_{k-1},a_k
  \rangle_\Y\\
   &\leq \langle p_k-(1-\theta)p_{k-1},q_k\rangle_\Y-(1/2)\|a_k\|_{\Y,G}^2+ (1/2)\|a_{k-1}\|_{\Y,G}^2
\end{align*}
where the last inequality is due to fact that $ 2 \langle Ga_k,a_{k-1}\rangle_\Y \leq \|a_k\|_{\Y,G}^2+\|a_{k-1}\|_{\Y,G}^2 $.
Therefore, (b) follows immediately from the last inequality, and then the proof is concluded.
 \hfill{ $\square$}



\def\cprime{$'$}

\end{document}